\title{Embeddability in $\R^3$ is NP-hard\thanks{An extended abstract of this
work appeared in the Proceedings of the Twenty-Ninth Annual ACM-SIAM Symposium on Discrete Algorithms~\cite{soda}.
AdM is partially
supported by the French ANR project ANR-16-CE40-0009-01 (GATO). AdM and MT are
partially supported by the Czech-French collaboration project EMBEDS II (CZ:
7AMB17FR029, FR: 38087RM). This work was partially supported by a grant from
the Simons Foundation (grant number 283495 to Yo'av Rieck).
MT is supported by the GA\v{C}R grant 16-01602Y and by Charles
University project UNCE/SCI/004.}}

\documentclass[11pt,a4paper]{article}
\usepackage{a4wide}

\usepackage{latexsym}
\usepackage{amssymb,amsmath}
\usepackage{graphicx}
\usepackage{url}
\usepackage{color}
\usepackage{verbatim}
\usepackage{epsfig}
\usepackage{enumerate}
\usepackage{hyperref}
\usepackage{environ}

\usepackage{authblk}
\author[1]{Arnaud de Mesmay}
\author[2]{Yo'av Rieck}
\author[3]{Eric Sedgwick}
\author[4]{Martin Tancer}
\affil[1]{\small{Univ. Grenoble Alpes, CNRS, Grenoble INP\footnote{Institute of Engineering Univ. Grenoble Alpes}, GIPSA-lab, 38000 Grenoble, France}}
\affil[2]{\small{Department of Mathematical Sciences, University of Arkansas
Fayetteville, AR 72701, USA}}
\affil[3]{\small{School of Computing, DePaul University, 243 S.~Wabash Ave,
Chicago, IL 60604, USA}}
\affil[4]{\small Department
of Applied Mathematics,
Charles University, Malostransk\'{e} n\'{a}m.
25, 118~00~~Praha~1, Czech Republic.}

\date{\today}

\newcommand{\NP}{\textbf{NP}}

\newcommand{\R}{{\mathbb{R}}}
\newcommand{\Q}{{\mathbb{Q}}}

\DeclareMathOperator{\lk}{lk}
\DeclareMathOperator{\interior}{int}
\DeclareMathOperator{\TRUE}{TRUE}
\DeclareMathOperator{\FALSE}{FALSE}
\DeclareMathOperator{\emptlk}{`empty\ link'}

\usepackage{amsthm}
\usepackage{thmtools}
\usepackage{thm-restate}
\theoremstyle{plain}
\newcommand{\thm}{theorem}
\newcommand{\lemm}{lemma}
\newcommand{\prop}{proposition}
\newtheorem{theorem}{Theorem}[section]

\newtheorem{lemma}[theorem]{Lemma}
\newtheorem{cor}[theorem]{Corollary}
\newtheorem{claim}[theorem]{Claim}

\theoremstyle{definition}
\newtheorem{notation}[theorem]{Notation}
\newtheorem{definition}[theorem]{Definition}

\theoremstyle{definition}
\newtheorem{remark}[theorem]{Remark}
\newtheorem{example}[theorem]{Example}

\long\def\onefigure#1#2{
\begin{figure*}[tbp]
\begin{center}
#1
\end{center}
\caption{#2}
\end{figure*}
}

\def\immediateFigure#1{%
\smallskip\begin{center}#1\end{center}\smallskip }

\newcommand{\labfig}[2]  
{\onefigure{\mbox{\includegraphics{#1}}}{\label{f:#1} #2} }

\newcommand{\labfigw}[3]  
{\onefigure{\mbox{\includegraphics[width=#2]{#1}}}{\label{f:#1} #3}}

\newcommand{\immfig}[1]  
{\immediateFigure{\mbox{\includegraphics{#1}}}}

\newcommand{\immfigw}[2] 
{\immediateFigure{\mbox{\includegraphics[width=#2]{#1}}}}

\definecolor{orange}{rgb}{1,0.5,0}

\begin{document}

\maketitle

\begin{abstract}
We prove that the problem of deciding whether a $2$- or $3$-dimensional simplicial complex embeds into $\mathbb{R}^3$ is \NP-hard.
Our construction also shows that deciding whether a \(3\)-manifold with
boundary tori admits an \(\mathbb{S}^{3}\) filling is \NP-hard.
The former stands in contrast with the lower dimensional cases which can be solved in linear time,
and the latter with a variety of computational problems in $3$-manifold topology (for example, unknot or $3$-sphere recognition, which are in $ \NP \cap \hbox{co-}\NP$ assuming the Generalized Riemann Hypothesis).
Our reduction encodes a satisfiability instance into the embeddability problem of a $3$-manifold with boundary tori, and relies extensively on techniques from low-dimensional topology, most importantly Dehn fillings on link complements.

\end{abstract}

\section{Introduction}

For integers $d \geq k \geq 1$, let $\textsc{Embed}_{k\rightarrow d}$ be the
algorithmic problem of determining whether a given $k$-dimensional simplicial
complex\footnote{For completeness we give the definition of a complex.
An \em \(i\)-simplex \em (or simply a simplex) is the convex hull of \(i+1\) generic points,
called \em vertices, \em in $\R^m$ for some integer $m$
(by generic we mean that the points are not contained in an affine space of dimension \(i-1\)).
A face of the simplex is a convex hull of any subset of its vertices.
A \emph{(geometric) simplicial complex $K$} is a collection of
simplices in $\R^m$ 
such that (i) if $\sigma$ belongs to $K$, then
any face of $\sigma$ belongs to $K$ as well; and (ii) if $\sigma$ and $\tau$
belong to $K$, then $\sigma \cap \tau$ is a face of both. The \emph{dimension}
of $K$ is the maximum of the dimensions of simplices in $K$. When considering
embeddability of $K$, we mean embeddability of the \emph{polyhedron} $\bigcup
\{\sigma \in K\}$. For computational complexity questions, a geometric
simplicial complex can be `stored' as an abstract simplicial complex containing
only the information which vertices span a simplex in $K$. These data determine
the polyhedron of $K$ uniquely up to homeomorphism. For further reading on
simplicial complexes we refer, for example, to~\cite{matousek03,munkres84}.}
embeds piecewise-linearly in $\mathbb{R}^d$. This problem generalizes
graph planarity, which corresponds to $k=1$ and $d=2$ and can be solved in
linear time (see~\cite{p-pte-16}).
The case $k=d=2$ is also known to be
decidable in linear time using similar techniques~\cite{gr-ltpa2c-79}.
In the past years, several results have appeared studying the computational complexity
of $\textsc{Embed}_{k\rightarrow d}$ for higher values of $k$ and $d$,
exhibiting very different behaviors depending on the relative values of $k$ and
$d$.
In dimension $d \geq 4$, the problem is polynomial-time decidable for $k <
(2d-2)/3$~\cite{ckv-aslep-17} (building
on~\cite{ckmsvw-camis-14,ckmvw-ptchgp-14,kms-pthsem-13}). In these polynomial
cases, known geometric tools (e.g. the Whitney trick and the Haefliger-Weber
theorem) allow to reduce the problem of embeddability to purely
homotopy-theoretical questions, which are then solved using techniques of
\emph{computational homotopy theory}. However, the Haefliger-Weber theorem
generally requires a high codimension to be executed, which typically fails for
low values of $d$.
In the remaining cases in dimension $d \geq 4$, that is, when $k \geq (2d-2)/3$,
the problem is \NP-hard~\cite{mtw-hescr-11}.
Decidability is not known for these \NP-hard cases, except for \(d \geq 5\) and \(k=d\) or \(d-1\),
which are known to be undecidable.
The cases $d=3$ and $k=2,3$ were the first intriguing gaps
left by these results, and the corresponding problems
$\textsc{Embed}_{2\rightarrow 3}$ and $\textsc{Embed}_{3\rightarrow 3}$ were
recently proved to be decidable by Matou\v{s}ek, Sedgwick, Tancer and
Wagner~\cite{mstw-e3sd-14} using an entirely different set of tools than the
other cases. Indeed, they rely extensively on techniques developed
specifically to study the topology of knots and $3$-manifolds, and in
particular on \emph{normal surface theory}.

Normal surfaces are a ubiquitous tool to solve decision problems in
$3$-dimensional topology, and are used in most decision algorithms in
that realm (see for example Matveev~\cite{m-atc3m-03}). However, algorithms relying on normal surfaces generally
proceed by enumerating big solution spaces to find ``interesting''
surfaces, which makes most of them remarkably
inefficient (at least theoretically), with upper bounds on the
runtimes ranging from singly exponential, for example for unknot
recognition~\cite{hlp-ccklp-99}, to merely elementary recursive for $3$-manifold
homeomorphism~\cite{k-ah3mcg-15}. The aforementioned algorithm of Matou\v{s}ek, Sedgwick,
Tancer and Wagner for $\textsc{Embed}_{2\rightarrow 3}$ is no
exception, proving a bound on the runtime that is at least an iterated exponential tower. For some problems, this inefficiency is somewhat justified by hardness
results: this has been an active area of research in recent years~\cite{aht-cckgs-06,bcm-cins-16,bds-chgnph-16,bmw-fnos3m-17,bs-cdtast-13,l-chpl3m-16}. In
particular, Burton, de Mesmay and Wagner~\cite{bmw-fnos3m-17} have shown that the
embeddability of non-orientable surfaces into $3$-manifolds
is \NP-hard (when the $3$-manifold is part of the input). In their
reduction, most of the complexity is encoded in the target
$3$-manifold, as is the case for most of the \NP-hardness results in
$3$-manifold theory (except for two recent hardness results on
classical links by Lackenby~\cite{l-chpl3m-16}). Furthermore, two of the most iconic
$3$-dimensional problems revolving around $\mathbb{S}^3$ or
submanifolds thereof, unknot and $3$-sphere recognition, are
expected \emph{not} to be \NP-hard since they are both in \NP~\cite{hlp-ccklp-99,i-r3s-01,s-srlnp-11} and
co-\NP~\cite{l-ecktn-16,z-ih3sai-16} (note that the co-\NP\  membership for $3$-sphere recognition
assumes the Generalized Riemann Hypothesis). This could give the
impression that $\mathbb{S}^3$ occupies a particular position in this
computational landscape, where problems tend to be easier than in the
general case.

\paragraph*{Our results.} In this article, we undermine 
this idea by proving that testing
embeddability into $\mathbb{R}^3$ is \NP-hard. Precisely, we
prove \NP-hardness of $\textsc{Embed}_{2\rightarrow 3}$ and
$\textsc{Embed}_{3\rightarrow 3}$, as well as for
$\textsc{$3$-Manifold Embeds in $\mathbb{R}^3$}$, the restriction of
$\textsc{Embed}_{3\rightarrow 3}$
when the domain is a $3$-manifold and the
range is $\mathbb{R}^3$.

\begin{\thm}~\label{T:main}
The following problems are all \NP-hard: $\textsc{Embed}_{2\rightarrow 3}$, $\textsc{Embed}_{3\rightarrow 3}$, and
$\textsc{3-Manifold  Embeds  in $\mathbb{R}^3$}$.
\end{\thm}

We observe that $\textsc{Embed}_{3\rightarrow 3}$ reduces to
$\textsc{Embed}_{2\rightarrow 3}$ following the algorithm described
in~\cite[Section 12]{mstw-e3sd-14}, and that
$\textsc{Embed}_{3\rightarrow 3}$ contains $\textsc{3-Manifold Embeds
in $\mathbb{R}^3$}$, so in order to prove Theorem~\ref{T:main}, it is enough to
prove that $\textsc{3-Manifold Embeds in $\mathbb{R}^3$}$ is \NP-hard. Since embeddability in $\mathbb{R}^3$ is equivalent to embeddability in $\mathbb{S}^3$ (unless the input complex is $\mathbb{S}^3$, which will never be the case in our reduction), we will work throughout the paper with $\mathbb{S}^3$ instead of $\mathbb{R}^3$.

In fact we get more.  On the one hand, we only need to consider embeddings of
complexes that are connected orientable 3-manifolds whose boundaries consists of tori (which are often seen as ``simpler'' than 3-manifolds with arbitrary boundary, let alone general complexes), and on the other hand, we may replace \(\mathbb{S}^{3}\) and consider embeddings into any fixed triangulated
 closed irreducible orientable 3-manifold \(M\) admitting no essential torus (these terms
are defined in the next section). i These 3-manifolds include all
closed orientable \em hyperbolic \em 3-manifolds, a family that is of
great interest.  We state this here as a corollary, and provide the
proof in the last section, Section~\ref{section:ProofOfCorollary}:

\begin{cor}
\label{corollary}
Let \(M\) be a triangulated closed orientable irreducible and atoroidal  3-manifold and let \(\mathcal{M}_{\mathrm{tor}}\) be the set of triangulated connected orientable 3-manifolds whose boundaries consists of a (possibly empty) collection of tori.  Then the decision problem:
\begin{quote}
Given \(N \in \mathcal{M}_{\mathrm{tor}}\), does \(N\) embed in \(M\)?
\end{quote}
is \NP-hard.
\end{cor}

\paragraph*{Outline and techniques.}The idea of the proof of Theorem~\ref{T:main} is to work with
3-manifolds with boundary tori, such as complements of $1$-manifolds
in $\mathbb{S}^3$, i.e., knot and link complements.  By the Fox
Re-embedding Theorem~\cite{fox}, if such a 3-manifold $M$ embeds in
$\mathbb{S}^3$, then there exists an embedding where
$\overline{\mathbb{S}^3 \setminus M}$ is a collection of solid
tori.\footnote{We only need the Fox Re-embedding Theorem for
3-manifolds with boundary tori which is a special case of
Lemma~\ref{lemma:reembading} that we prove in
Section~\ref{section:ProofOfCorollary}; see Remark~\ref{remark:fox}.} 
 The process of filling a torus boundary with a solid torus is called
a \em Dehn filling\em, and thus, deciding whether these 3-manifolds
embed in $\mathbb{S}^3$ amounts to understanding whether one can
obtain $\mathbb{S}^3$ with a good choice of Dehn fillings.  By a
celebrated theorem of Gordon and Luecke~\cite{gordon-luecke}, the
situation is very constrained when considering knot complements, since
they prove that the only Dehn filling on a nontrivial knot in $\mathbb{S}^3$  yielding $\mathbb{S}^3$  is the trivial one. This was used
by Jaco and Sedgwick~\cite{js-dpsdf-03} to provide an algorithm to
recognize knot complements, which is equivalent to testing the
embeddability of a $3$-manifold with a single torus boundary into
$\mathbb{S}^3$.

By contrast, the situation is far richer for link complements, because one can easily produce links where distinct Dehn fillings all give $\mathbb{S}^3$ (see~\cite{RieckYamashita} for a variety of examples and an attempt at constraining such fillings). Our hardness proof leverages this complexity, by encoding a SAT instance within a link with non-trivial Dehn fillings already carried out on some of its components, yielding a $3$-manifold $M$, such that the instance is satisfiable if and only if $M$ is embeddable in $\mathbb{S}^3$. The idea behind our reduction is rather simple, and we present a simplified version of it as a warm-up at the start of Section~\ref{S:reduction}, but proving that it works, and in particular that there are no \em accidental \em embeddings into $\mathbb{S}^3$, requires significant work, relying on multiple prior results on Dehn fillings~\cite{cgls,scharlemannDiskSphere}, cable spaces~\cite{gordon-litherland} and knotted handlebodies
(see, for example,~\cite{i-tgthku-12} for a discussion).

\section{Preliminaries}
In these preliminaries, we introduce most of the notions that we will use, but in
the rest of the paper, especially Sections~\ref{S:otherdirection}
and~\ref{sec:VxLxBoundaryIrreducible},  we assume some familiarity with knot theory and $3$-manifold topology. Good references for these are Rolfsen~\cite{Ro90} and Schultens~\cite{s-it3m-14}.

\subsection{Dehn surgery on knots and links}
In our exposition we mainly follow~\cite{Ro90}.
We work in the PL category.   By \em 3-manifold \em we mean compact connected orientable \(3\)-dimensional manifold.
A 3-manifold is called \em closed \em if its boundary is empty.
We use standard notation: \(\partial\), \(\mathrm{int}\), and \(N\)
stand for boundary, interior,  and (closed) normal neighborhood, respectively.
The closure of a set \(X\) is denoted \(\overline{X}\).
We write \(\mathbb{S}^{n}\) for the \(n\)-sphere and \(\mathbb{D}^{2}\) for the closed 2-disk.
By \em sphere \em we mean \(\mathbb{S}^{2}\).
We always assume general position.
When considering homology we always use integral coefficients, that is,
by \(H_{1}(X)\) we mean \(H_{1}(X;\mathbb{Z})\).

\paragraph{Compressions and reductions.}
A curve embedded in a surface $F$ is a connected $1$-manifold embedded in \(F\), which is either a \em loop \em (if it is closed) or an \em arc \em with two endpoints on $\partial F$. A curve is \em inessential \em if it is a loop bounding a disk or an arc cobounding a disk with some arc in $\partial F$, \em essential \em otherwise.

A sphere $S$ in a $3$-manifold $M$ is a \em reducing sphere \em if $S$ does not
bound a ball in $M$. A $3$-manifold without reducing spheres is called \em
irreducible\em.  Let $F$ be a surface properly embedded in $M$ or embedded in
$\partial M$. A \em compressing disk \em for $F$ is an embedded disk $D
\subseteq M$ whose interior is disjoint from $F$ and whose boundary is an
essential loop in $F$. A \em boundary compressing disk \em is an embedded disk
$D \subseteq M$ whose interior is disjoint from $F \cup \partial M$,
and whose boundary
$\partial D=f \cup x$ is the union of two arcs, where
$f=\partial D \cap F=D \cap F$ is an
essential arc properly embedded in $F$, and $x=\partial D \cap \partial M= D
\cap \partial M \subset \partial M$. A surface is \em compressible \em if it has
a compressing disk or is a $2$-sphere bounding a ball, \em boundary
compressible \em if it has a boundary compressing disk or is a disk
cobounding a ball with a disk in $\partial M$, and \em
incompressible \em and \em boundary incompressible \em otherwise. A
surface $F \subseteq M$ is \em boundary parallel \em if it is
separating and a component of $M \setminus F$ is homeomorphic to
$F \times I$ with \(F\) corresponding to $F \times \{0\}$
(here \(I=[0,1]\)).  Finally, a surface is \em essential \em if it is
incompressible, boundary incompressible and not boundary parallel. A
$3$-manifold $M$ is \em boundary irreducible \em if its boundary is
incompressible.

\paragraph{The solid torus and genus 2 handlebody.}
The solid torus is \(\mathbb{D}^{2} \times \mathbb{S}^{1}\).  Since it plays a key role in our game
we mention a few facts about it here.  The isotopy class\footnote{For a discussion
of \em isotopy \em of curves see~\cite[Chapter~1]{Ro90}.}
of a curve \(\{\mathrm{pt}\} \times \mathbb{S}^{1}\)
is called the \em core \em of the solid torus.
The disk \(\mathbb{D}^{2} \times \{\mathrm{pt}\}\) is called the \em meridian disk \em
of  \(\mathbb{D}^{2} \times \mathbb{S}^{1}\).  The isotopy class of
boundary of the meridian disk in the boundary of the solid torus,
\(\mathbb{S}^{1} \times \{\mathrm{pt}\} \subset \partial \mathbb{D}^{2} \times \mathbb{S}^{1}\), is called the
\em meridian \em of the solid torus.
We emphasize that we consider the meridian as a curve in the boundary torus and \em not \em in the solid torus.
The boundary of the solid torus is, of course, a torus, and the homology
class of the meridian disk generates the kernel of the homomorphism
\(H_{1}(\partial \mathbb{D}^{2} \times \mathbb{S}^{1}) (\cong \mathbb{Z}^{2})
\to H_{1}(\mathbb{D}^{2} \times \mathbb{S}^{1}) (\cong \mathbb{Z})\)
induced by the inclusion.

Note that if \(\gamma\) is a circle embedded in a 3-manifold  then \(N(\gamma)\) is a solid torus (recall that we only consider orientable 3-manifolds).
Similarly, if \(A\) is an annulus or a M\"obius band embedded in a 3-manifold then \(N(A)\)
is a solid torus.
If \(\gamma\) is the wedge of two circles (that is, the
union of two circles
whose intersection is exactly one point) embedded in a 3-manifold, we call \(N(\gamma)\) the \em genus 2 handlebody. \em
It is easy to see that the genus \(2\) handlebody is well defined, that is, the homeomorphism type
of \(N(\gamma)\) does not depend on the choices made.  In fact, the genus 2 handlebody is
homeomorphic to the 3-manifolds obtained by gluing two solid tori along a disk in their boundaries.

\paragraph{Knots and links.} By a \emph{knot} \(\kappa\) in a 3-manifold \(M\)
(typically \(M = \mathbb{S}^{3}\))
we mean an isotopy class of an embedding
\(\kappa:\mathbb{S}^{1} \to M\); as is customary we identify a knot with its image,
and write \(\kappa \subset M\) for \(\kappa(\mathbb{S}^{1}) \subset M\)
and \(N(\kappa)\) for a regular neighborhood of \(\kappa\) (that is, \(N(\kappa)\) is a solid torus embedded in
\(M\) and \(\kappa\) is its core).
A knot \(\kappa \subset \mathbb{S}^{3}\) is the \emph{unknot} if $\kappa$ is the boundary of an embedded disk in $\mathbb{S}^3$.
A \emph{link} \(L\) is a collection of pairwise disjoint knots, called the \em components \em of \(L\).
Links are considered up to isotopy during which the components are required to remain pairwise disjoint.
A link \(L\) is \emph{split}   
if there is an embedded $2$-sphere disjoint from \(L\), called a \emph{splitting sphere},
separating some components from others.
If there is a splitting sphere for \(\kappa_{1} \sqcup \kappa_{2}\) separating the two components
we say that they are \emph{unlinked}, \emph{linked} otherwise.
An \emph{unlink} is a link in which every component can be split form all remaining components and
every component is unknotted (that is, there exists disjointly embedded balls so that each component is an unknot
contained in its own ball).
The \em exterior \em of a link \(L \subset M\) is \(E(L) := \overline{M \setminus \mathrm{int}N(L)}\).

We view \( \mathbb{S}^{3}\) as \(\mathbb{R}^{3} \cup \{\infty\}\).
A \emph{link diagram} for a link
\(L \subset \mathbb{S}^{3}\) is a projection of \(L\) to the plane $\R^2$ induced by \((x,y,z) \to (x,y)\) such that \(\infty \not\in L\),
every point \((x,y)  \in \R^2\) has at most two preimages on the link,
and furthermore,  whenever $(x,y)$ has two preimages, they correspond to
a transversal crossing at $(x,y)$ of two subcurves of the link.
These conditions are satisfied if \(L\) is in general position.
A point with two preimages is called a \em crossing\em, and we mark which subcurve passes over and which under.   Note that a link diagram
determines a unique link in $\mathbb{S}^3$.  We will exploit this: in order to
construct a link in \(\mathbb{S}^{3}\) we will construct, directly, its
diagram.

\paragraph{Dehn filling and Dehn surgery.}

We now explain the concept of \em Dehn filling\em, and the closely related concept
of \em Dehn surgery\em.  Later on, in Section~\ref{S:otherdirection},
we will show that these two concepts control
embeddability into \(\mathbb{S}^{3}\)  in our setting, and for that reason they are central to our work.
Let \(M\) be a 3-manifold and \(T \subset \partial M\) a torus.

By \em Dehn filling \em
\(T\) we mean attaching a solid torus \(\mathbb{D}^{2} \times \mathbb{S}^{1}\)
to \(T\).  (Note that the boundary of the resulting manifold is $\partial M
\setminus T$.) However, the result of Dehn filling is not uniquely determined by
\(M\) and \(T\), as it depends on the choice of attachment (that is, the choice
of homeomorphism \(\partial (\mathbb{D}^{2} \times \mathbb{S}^{1}) \to T\)); see Figure~\ref{f:slope}.
It can be shown\footnote{For this, and many other facts about Dehn surgery;
see~\cite[Section~9.F]{Ro90}.} that up to homeomorphism the resulting
3-manifold is determined by the homology class $r$ of the image of the meridian of the
attached solid torus in \(H_{1}(T)\); such a class is called a \em slope \em of
\(T\). Then by \(M(r)\) we denote the result of Dehn filling \(T\) along the
slope \(r\).

Formally, a \em slope \em of \(T\) is an element of \(H_{1}(T)\) that is represented by an
embedded circle in \(T\).  (Algebraically, slopes can be characterized  as the \em primitive \em elements of \(H_{1}(T)\),
that is, elements that together with one other element generate  \(H_{1}(T)\).) When the manifold $M$ is $\mathbb{S}^3$, we can parameterize these slopes by rational numbers, this will be explained below.

\begin{figure}
\begin{center}
\includegraphics{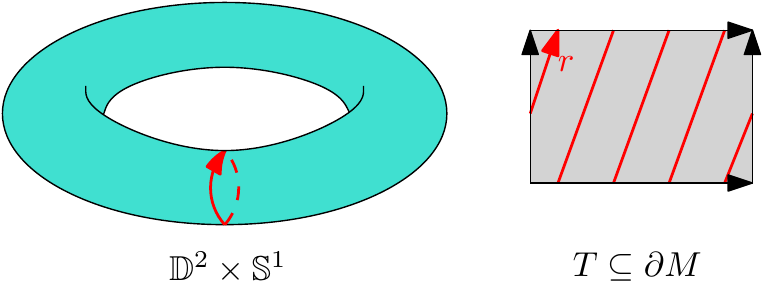}
\caption{Dehn filling: The meridian of the solid torus \(\mathbb{D}^{2} \times
\mathbb{S}^{1}\) is glued to a (representative of a) slope $r$ in $T$.}
\label{f:slope}
\end{center}
\end{figure}

A \em Dehn surgery \em on a knot \(\kappa \subset M\) is  Dehn filling on the component \(\partial N(\kappa)\) of the \(\partial E(\kappa)\).  Note that since \(\partial M\) may not be empty,  \(\partial E(\kappa)\) may have other components; by Dehn surgery on \(\kappa\) we mean Dehn filling, specifically,
the component of \(\partial E(\kappa)\) that corresponds to \(\kappa\), that is, \(\partial E(\kappa) \setminus \partial M\). Dehn surgery on \(\kappa\) is determined
by a choice of slope on \(\partial N(\kappa)\).  As is customary we refer to this slope
as a \em slope of \em \(\kappa\).

\begin{example}[Simple but instructive]
\label{example:DehnFillingSolidTorus}
Let \(M \cong \mathbb{S}^{1} \times
\mathbb{D}^{2}\) be the solid torus.  We consider the manifolds $M(r)$ obtained by
filling the slope $r$ in the boundary of \(M\).  By definition, these are exactly the
manifolds that can be obtained as the union of two solid tori.    An easy
calculation (see, for example,~\cite[Example~2.43]{hatcher02}) shows that the manifold $M(r)$ resulting from filling has homology group \(H_{1}(M(r)) \cong \mathbb{Z}_{|r|}\), where $|r|$ is the
minimal number of times the boundary of the meridian disk of the attached
solid torus meets the boundary of the meridian disk of the solid torus $M$.
Moreover, $M(r) \cong \mathbb{S}^3$  if and only if $|r|=1$. 

This simple example shows that by changing the slope \(r\) we can change the homology $H_1(M(r))$ and thus obtain infinitely many distinct
manifolds as \(M(r)\) (in fact this is always the case).
On the other hand, since there are infinitely many slopes that satisfy \(|r| = 1\) the same manifold (in this case $\mathbb{S}^3$) can be obtained
infinitely many times as $M(r)$ (in fact, this is almost never the case and obtaining the same manifold more than
once is an interesting question that is not fully understood).
\end{example}

\paragraph{Parametrization of slopes.}   While the concept of Dehn surgery is more general, for our purposes, it is sufficient to consider only surgery on a link  \(L \subset \mathbb{S}^{3}\), or a link $L \subset M$, where $M$ is a 3-manifold embedded in $\mathbb{S}^3$.   Since every component
\(\kappa\) of \(L\) is a knot in \(\mathbb{S}^{3}\), it has two distinguished slopes called the \em meridian \em
and \em longitude\em, which intersect once and together form a basis for \(H_{1}(\partial N(\kappa))\).  The \em meridian \em is the unique essential
simple closed curve in
$\partial N(\kappa)$ (up to isotopy) that bounds a disk in the solid torus $N(\kappa)$.  Surgery along
this slope does not change the 3-manifold: we reattach \(N(\kappa)\) exactly
as it was prior to removing it.
The \em longitude \em
is the unique essential simple closed curve in $\partial N(\kappa)$ (up to isotopy)
that bounds an orientable surface in the exterior of the knot $E(\kappa) = \mathbb{S}^3 \setminus \interior N(\kappa)$.  In this paper each component of our link is an unknot and,  in that case, this surface can always be taken to be the disk bounded by the knot (which is allowed to meet other components of the link in its interior).

Having the homology classes of the meridian and longitude in hand we may write them as
\((0,1)\) and \((1,0)\), respectively.  This gives a concrete identification of \(H_{1}(\partial N(\kappa))\) with
\(\mathbb{Z}^{2}\).\footnote{If \(r\) corresponds to \((q,p)\),
the numbers \(q\) and \(p\) have easy topological interpretations: \(q\) is the minimal
number of times that \(r\) intersects the meridian and \(p\) is the minimal
number of times that \(r\) intersects the longitude.}

The slopes of \(\partial N(\kappa)\) are represented by
embedded circles.
An element \((q,p) \in \mathbb{Z}^{2}\)
represents an
embedded circle
if and only if \(q\) and \(p\) are relatively prime.  Finally,
we note that \((q,p)\) and \((-q,-p)\) represent the same curve and hence the same slope.  It is therefore
natural to identify slopes on \(\partial N(\kappa)\) with
\[
\textit{slopes} = \mathbb{Q} \cup \{1/0\}
\]
by identifying \((q,p)\) with \(p/q\).
Note that under this identification the meridian is identified with \(1/0\) and the longitude with \(0/1\).

\paragraph{Dehn surgery on a link with coefficients.} With the above
parameterization for slopes, we are now prepared to introduce our principal
notation for a 3-manifold obtained by Dehn surgery on a link in a sub-manifold of
$\mathbb{S}^3$. For examples, see Figure
\ref{f:simplified_example},  where $M = \mathbb{S}^3$.   Each component of the link will be labeled with a \em surgery coefficient, \em an instruction for Dehn surgery.   In addition to slopes of \(\kappa\) we allow the surgery coefficient of \(\kappa\) to be \(\emptyset\).
We define surgery on \(\kappa\) with coefficient \(\emptyset\) to be $E(\kappa)$,
and call this process \em drilling \(\kappa\) out\em. Thus, we may identify all possible coefficients
with

\[\mathit{coefficients} = \mathbb Q \cup \{\emptyset, 1/0\}.\]

We summarize this construction and the notation we will use (often with
$M=\mathbb{S}^3$).

\begin{notation}
By a \em link with surgery coefficients \em (in $\mathbb{S}^3$) we mean a link
\(L \subset \mathbb{S}^{3}\) for which each component $\kappa$ is labeled with a
slope $r_\kappa \in \mathbb Q \cup \{\emptyset, 1/0\}$.
We emphasize that the notation \(L\) means the link including the coefficients.
Then \(M(L)\), the 3-manifold obtained by
Dehn surgery on \(L\) along the given slopes, is defined in the following way:
	\begin{enumerate}
	\item If \(r_{\kappa} = \emptyset\), we drill \(\kappa\) out.
	\item If \(r_{\kappa} \in \mathbb Q \cup \{1/0\}\), we perform a Dehn surgery
	with coefficient \(r_{\kappa}\).
	\end{enumerate}
\end{notation}

It is interesting to note that a result of Lickorish~\cite{lickorish, Ro90} states
that  every closed orientable 3-manifold has a description of this form, that
is, it can be written as $\mathbb{S}^3(L)$  for some link $L \subset \mathbb{S}^3$
with surgery coefficients.

\begin{example}[Simple but important]
\label{example:lensspace}
Consider the unknot $\kappa \subset \mathbb{S}^3$ with surgery coefficient $p/q$.

Note that the exterior of the unknot \(E(\kappa)\) is a solid torus.  Thus $\mathbb{S}^3(\kappa)$ is the union of two solid tori; or, as described in Example \ref{example:DehnFillingSolidTorus}, is a manifold $M(r)$ obtained by performing Dehn filling on the solid torus $M=E(\kappa)$.  As in that example,  $\mathbb{S}^3(\kappa)$ has homology \(H_{1}(\mathbb{S}^{3}(\kappa))
\cong \mathbb Z_{|r|}\) where $|r|$ is the minimal
  number of times the boundary of the meridian disk of the attached solid torus
  (slope $p/q$) meets the longitude of $\kappa$ (slope $0/1$).   Since that intersection number is
  $|p|$,  \(\mathbb{S}^{3}(\kappa)\) has homology \(H_{1}(\mathbb{S}^{3}(\kappa))
  \cong \mathbb{Z}_{|p|}\).  Moreover, \(\mathbb{S}^{3}(\kappa) \cong \mathbb{S}^{3}\) if
  and only if the two meridians intersect minimally exactly once, i.e. if and
  only if $|p| = 1$.
\end{example}

\paragraph{Modifications of a link with surgery coefficients.} We will utilize two operations on a link $L$, given by Propositions \ref{p:delete10} and \ref{p:simple_twist} below,  so that the resulting link $L'$ describes a surgered 3-manifold homeomorphic to the original, that is $\mathbb{S}^3(L') \cong \mathbb{S}^3(L)$.  (In fact, these operations are sufficient: the links $L$ and $L'$ describe the same 3-manifold $\mathbb{S}^3(L') \cong \mathbb{S}^3(L)$ only when $L$ and $L'$ are related by a sequence of these operations. This is the topic of the Kirby calculus~\cite{kirby}).

In order to state these operations, we first define the linking number of two
knots $\iota$ and $\kappa$:\footnote{We follow~\cite[5.D(3)]{Ro90}; in the same section Rolfsen
provides seven other equivalent definitions of the linking number.}
From now on we assume that $\mathbb{S}^3$ is equipped with a fixed orientation. We also
assume that $\iota$ and $\kappa$ are oriented. Considering a link
diagram for $\iota \cup \kappa$ we consider all crossings where $\iota$ passes
under $\kappa$ and
we assign them with $1$ or $-1$ according to the figure below.
\begin{center}
\includegraphics{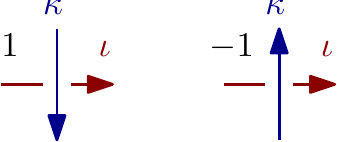}
\end{center}
The \emph{linking number} $\lk(\iota,\kappa)$ is then the sum of these assigned values.
It turns out that the value in the definition is independent of the choice of
the projection for obtaining the link diagram.
It also turns out that
$\lk(\iota,\kappa)
= \lk(\kappa,\iota)$. Note that a pair of knots can have linking number zero yet not be unlinked.

In~\cite[9.H]{Ro90} Rolfsen
describes a general recipe how to modify a link $L$ with surgery coefficients into a link $L'$
with (possibly different) surgery coefficients such that $\mathbb{S}^3(L)$ is
homeomorphic to $\mathbb{S}^3(L')$ (see namely Proposition~2
in~\cite[9.H]{Ro90}).
We have already noted that removing (hence adding) components with coefficient $1/0$ does not change the 3-manifold obtained by surgery:

\begin{\prop}
\label{p:delete10}
Let $L \subset \mathbb{S}^3$ be a link with surgery coefficients, and $L' = L \cup \kappa$, where $\kappa$ is any knot with surgery coefficient $1/0$.  Then $\mathbb{S}^3(L')$ is homeomorphic to $\mathbb{S}^3(L)$.
\end{\prop}

For the second operation, we need the following definition. Suppose one
component, say $\kappa_1$ of a link $L=\kappa_1 \cup \kappa_2 \cup \ldots \cup
\kappa_s$, is unknotted, and let $D$ be a disk that it bounds. Let $t$
be an integer. The link $L_t$
obtained by \emph{twisting $t$ times about $\kappa_1$} is the link obtained by
replacing locally the arcs of $\kappa_2, \ldots, \kappa_s$ going through $D$ by
$t$ helices which screw through a collar of $D$ in the right hand 
sense; see Figure~\ref{f:twisting}. (If $t$ is negative, we have
$|t|$ helices in the left hand sense.)

\begin{figure}
\begin{center}
\includegraphics[width=9cm]{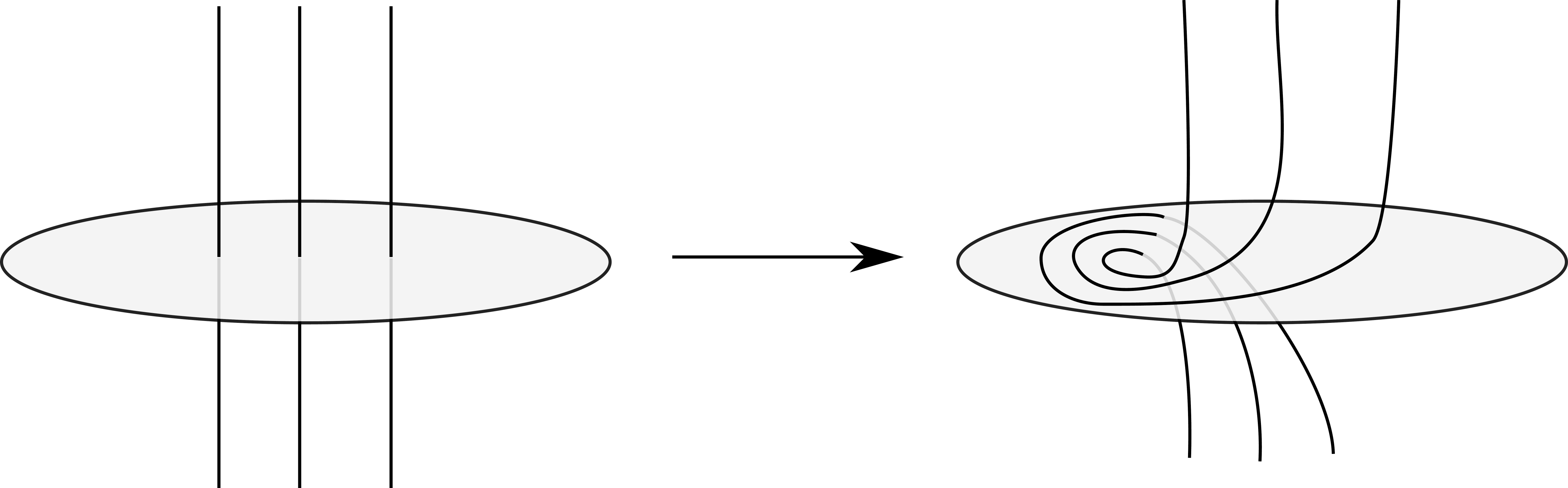}
\caption{Twisting once about an unknot.}
\label{f:twisting}
\end{center}
\end{figure}

\begin{\prop}
\label{p:simple_twist}
Let $L = \kappa_1 \cup \kappa_2 \cup \cdots \cup \kappa_s$ be a link
with surgery coefficients $r_1, r_2, \dots, r_s$, where
$s \geq 2$ and $r_1  \neq \emptyset$ and where $\kappa_1$ is an unknot.\footnote{Rolfsen in~\cite[9.H]{Ro90}
does not discuss the coefficients $\emptyset$. However, allowing some of the
coefficients $r_3, \cdots, r_s$ to be $\emptyset$ does not affect the proof
given in~\cite[9.H]{Ro90}.}

Let $t$ be an integer parameter, and let $L_t = \kappa_1 \cup \kappa_2' \cdots \cup
\cdots \kappa_s'$ be the link obtained by twisting $t$ times about
$\kappa_1$ with
surgery coefficients $r'_1, r'_2, \dots, r'_s$, where
\(r_{i}' = \emptyset\) whenever \(r_{i} = \emptyset\), and otherwise
\(r_{i}'\) is given by

$$ r'_i =
    \begin{cases}
      \frac{1}{t+1/r_1}       & \quad \text{if } i = 1;\\
          r_i + t (\lk(\kappa_i,\kappa_1))^2 & \quad \text{if } i>1;\\
        \end{cases}
$$
Then $\mathbb{S}^3(L)$ is homeomorphic to $\mathbb{S}^3(L')$. (Note,
in particular, that $r'_1 = 0$ if $r_1 = 0$, $r'_1 = 1/t$ if $r_1 = 1/0$, and
  $r'_1 = 1/0$ if $t + 1/r_1 =
0$.)
\end{\prop}

\begin{figure}
\begin{center}
\includegraphics{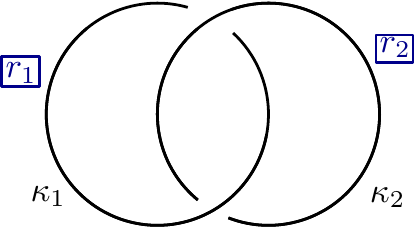}
\caption{The Hopf link (Proposition~\ref{p:simple_twist}, Example~\ref{e:linked_unknots})}
\label{f:basic_link}
\end{center}
\end{figure}

\begin{example} ~
\label{e:linked_unknots}
\begin{enumerate}[$(a)$]
\item
Let $L$ be the \em Hopf link\em, given in 
    Figure~\ref{f:basic_link}, with surgery
coefficients $r_1 = 1/n$ for some $n \in \mathbb{Z}$ and $r_2 = 1/0$ (that is, $\kappa_2$
could be removed). Then by setting $t = -n$ in Proposition~\ref{p:simple_twist},
we obtain $\mathbb{S}^3(L) = \mathbb{S}^3(\emptlk) = \mathbb{S}^3$.

\item
Let $L$ be the Hopf link,  this time with surgery  coefficients $r_1 = -1/k$ and $r_2 = \frac{1-kn}n$ for
  $k, n \in \mathbb{Z}$. Then still $\mathbb{S}^3(L)
  = \mathbb{S}^3(\emptlk) = \mathbb{S}^3$. Indeed, we first apply
  Proposition~\ref{p:simple_twist} with $t=k$ obtaining coefficients $r'_1 =
  1/0$ and $r'_2 = 1/n$. Now, we use part $(a)$ after swapping
  $\kappa_1$ and $\kappa_2$.

\end{enumerate}
\end{example}

We will see in Section~\ref{S:onedirection} a more intricate use of this proposition; see Figure~\ref{f:kirby}.

\subsection{NP-hardness}
\label{subsection:NPhard}

A \emph{Boolean formula} is a formula built from \emph{literals} (a variable or
its negation) using conjunctions and disjunctions. It is \emph{satisfiable} if
there exists a truth assignment for the variables making it true. It is in
\emph{conjunctive normal form} if it is a conjunction of disjunctions, such as
for example the formula $\Phi=(a \vee b) \wedge  (\neg b \vee c) \wedge (c \vee
a \vee \neg b)$. The disjunctions are called the \emph{clauses} of the formula.
It is a \emph{3-CNF} formula if every clause is made of at most $3$ literals.
The \emph{complexity} $|\Phi|$ of a 3-CNF formula $\Phi$ is its number of
literals plus its number of clauses. The \textsc{3-SAT} problem is the problem
of determining whether a given 3-CNF formula is satisfiable.

This problem is well known to be \textbf{NP}-hard (see for
example~\cite{ab-ccma-09}). Furthermore, the \textsc{3-SAT} problem remains
\textbf{NP}-hard when restricted to instances where (1) every clause contains
exactly 3 literals and (2) each clause contains at most once each variable; see for example Garey and
Johnson~\cite{gj-cigtnp-90}. In the rest of the text we will always consider
3-CNF formulae with these additional assumptions.

\section{The reduction}\label{S:reduction}
\renewcommand{\l}{{\pm}}
\newcommand{\lx}{{\pm x}}

\paragraph{Warm-up.} Before we present a formal description of our reduction,
let us start with a simplified construction that will reveal our main
idea. For simplicity, given that this is not our final construction,
most of the discussion here will be focused on an example coming from
a concrete formula.

\begin{figure}[h]
\begin{center}
\includegraphics{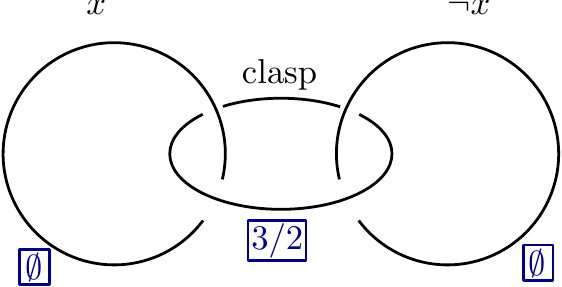}
\caption{The link corresponding to each variable.}
\label{f:variable_link}
\end{center}
\end{figure}

\begin{figure}[h]
\begin{center}
\includegraphics{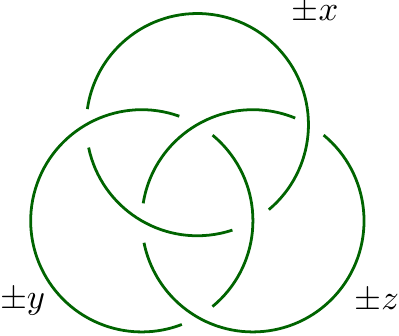}
\caption{The clause diagram - Borromean rings.}
\label{f:borromean}
\end{center}
\end{figure}

Starting from a formula $\Phi$, we build a link $L$ with surgery coefficients in the following way:
\begin{itemize}
\item The link $L$ contains one unknot for each literal of the formula (i.e., for each variable and its negation).
\item For each variable, the two unknots corresponding to it and its negation are interlinked with a \emph{clasp}, which is yet another unknot; see Figure~\ref{f:variable_link}.

\item For every clause of $\Phi$, we consider a triplet of Borromean rings (see Figure~\ref{f:borromean}). Each ring corresponds to a literal of the clause, and is `attached' to the unknot corresponding to its literal.
\item In this link, the surgery coefficients for the literals are $\emptyset$, and those for the clasps are $3/2$.
\end{itemize}

The precise definition of `attachment' (\emph{band sum}) will be given later on, but for now the following example and figure should give the reader a sufficient idea of the construction.

Let us consider the (satisfiable) formula $\Phi = (t \vee x \vee
y) \wedge (\neg x \vee y \vee z)$.
The link $L$
with surgery coefficients corresponding to this formula is depicted on Figure~\ref{f:simplified_example}.

\begin{figure*}[ht]
\begin{center}
\includegraphics{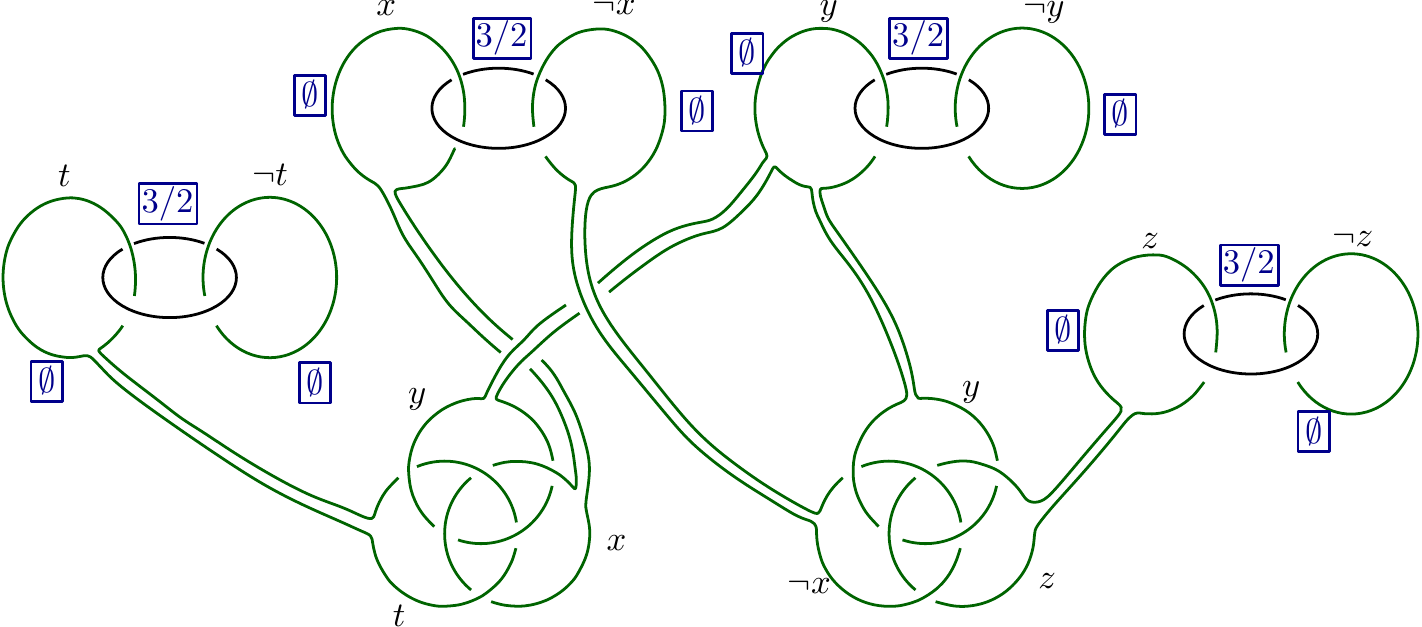}
\caption{ A simplified construction corresponding to the formula $\Phi = (t \vee x \vee
  y) \wedge (\neg x \vee y \vee z)$}
\label{f:simplified_example}
\end{center}
\end{figure*}

We consider the 3-manifold \(\mathbb{S}^{3}(L)\) obtained by Dehn
surgery on \(L\) with the given surgery coefficients. The key feature
of this construction is that a satisfying assignment for $\Phi$
directly reveals that $\mathbb{S}^3(L)$ embeds in
$\mathbb{S}^3$. Indeed, we can obtain the embedding with the following
simple rule: for each $\TRUE$ variable $v$, we fill the unknot
corresponding to $v$ with the slope $1/0$, while for each $\FALSE$
variable $v'$, we fill the unknot corresponding to $\neg v'$ with the
slope $1/0$. All the other unknots are then filled with coefficients
$1/1$. We claim that these fillings yield a $3$-sphere. We explain why
on the example introduced above, and the same argument easily applies
for the general construction.

Let us consider the satisfying
assignment $t = \TRUE$, $x, y, z = \FALSE$. Consequently, we fill the
unknots for the literals $t$, $\neg x$, $\neg y$ and $\neg z$ with
coefficient $1/0$, that is, we remove them from the link. 
See Figure~\ref{f:example_removing}.

\begin{figure*}[ht]
\begin{center}
\includegraphics{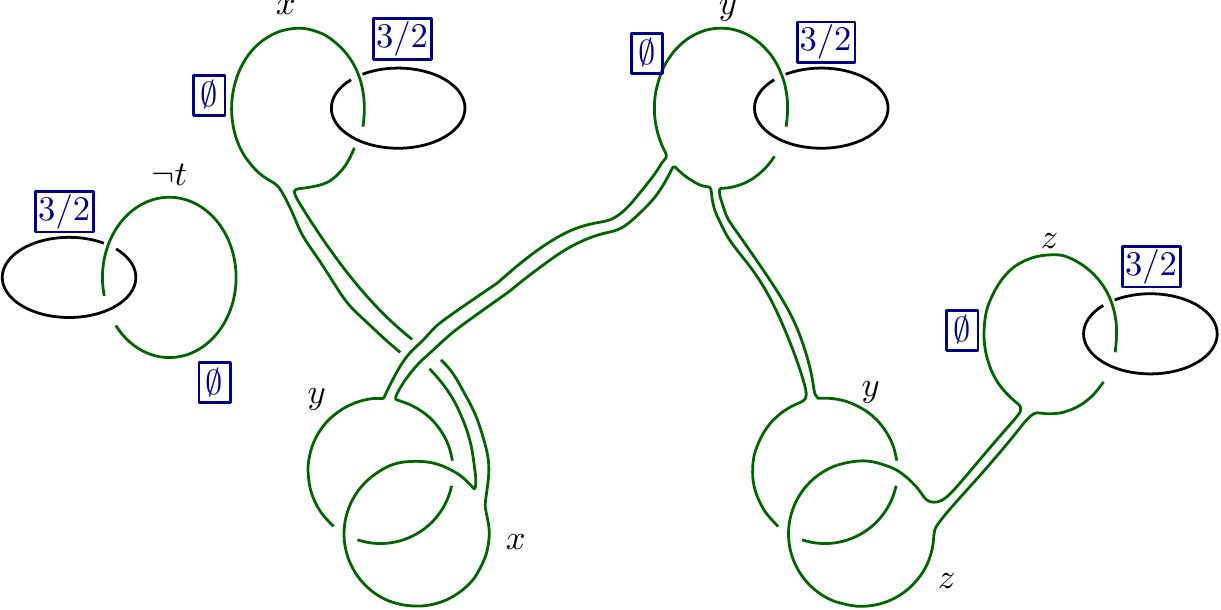}
\caption{Removing unknots with respect to a satisfying assignment.}
\label{f:example_removing}
\end{center}
\end{figure*}

Now, we fill the remaining unknots for literals with coefficients $1/1$.
Because the initial assignment is satisfying, all Borromean rings unlink and we
get in our case four unlinked copies of the Hopf
link, each contained in a ball and with coefficients \(3/2\) and \(1/1\).
This yields \(\mathbb{S}^{3}\) by Example~\ref{e:linked_unknots}~(b)
with $k=-1$ and $n=2$.

The hard part of our reduction is to show that this simple rule,
relying on a satisfying assignment to fill unknots with $1/0$ or
$1/1$, is `essentially' the only way we may get $\mathbb{S}^3$. In
particular we would like to show that we cannot get $\mathbb{S}^3$ for
a non-satisfiable formula.

In fact, we are not sure whether this claim is true for the simple construction
above. However, once we modify the construction slightly, we are able to show
that the resulting 3-manifold embeds into $\mathbb{S}^3$ if and only if it comes from a
satisfiable formula. Now we proceed with a formal description of our final
construction.

\paragraph{Full construction.}

Given a 3-SAT formula $\Phi$, satisfying the conditions stated in Subsection~\ref{subsection:NPhard}, we construct a 3-manifold, $M = \mathbb{S}^3(L)$, where $L \subset \mathbb{S}^3$ is a link  that is described by a planar diagram; our construction will produce the diagram (including surgery coefficients) explicitly.
In order to guarantee that the 3-manifold does not embed when the formula is
not satisfiable, we complicate the construction from the warm-up in two
ways:~(1) we further entangle each clasp and its literals, and~(2) we replace
each literal component with its $(2,1)$-{\em cable}, i.e. a component that
follows twice along and twists once around the original.  We now describe the
construction in detail.

\begin{figure}[h]
\begin{center}
\includegraphics{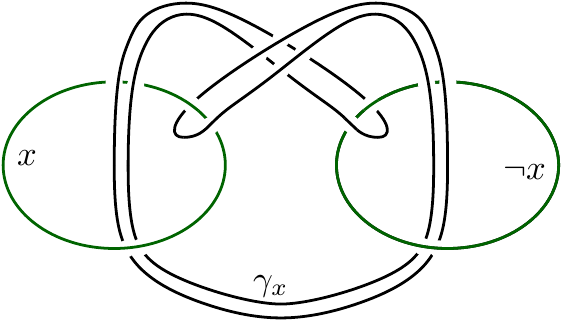}
\caption{The variable diagram.}
\label{f:variable_full}
\end{center}
\end{figure}

\paragraph{Variable Diagram.}
For each variable $x$, place a copy of the diagram depicted in 
Figure~\ref{f:variable_full} 
in the plane.  The diagram depicts a link with three unknotted components.  Label the left and right components with $x$ and $\neg x$, respectively.   We will refer to the central component as the {\em clasp} $\gamma_x$.

If $x$ is a variable, we will let $\lx$ denote either the literal $x$ or the literal  $\neg x$.

\paragraph{Clause Diagram.} For each clause $\pm x \vee \pm y \vee \pm z$ in
the formula, embed a diagram of the Borromean rings in the plane; see Figure~\ref{f:borromean}. 
The properties of the Borromean rings that
we need are (1)~each component is an unknot, and
(2)~removing any component results in a two component link
where the components are not linked.

\paragraph{Connecting the diagram.}
For each literal $\lx$ occurring in a clause, identify (but
don't draw) an embedded arc connecting the literal from the clause diagram to the
literal in variable diagram; see Figure \ref{f:construction_full1}.   These arcs can be chosen so that: a) the interior of each arc is disjoint from every variable and clause diagram, and b) each pair of arcs meet in at most one point and, when so, that point is in the interior of each arc.
Whenever two such arcs
cross, we arbitrarily pick which one lies above the other one.
Note that at this point we have exactly \(16\) crossings in each variable
diagram, \(6\) crossings in each clause diagram, and at most
\({3n}\choose{2}\) crossings between connecting arcs, where \(n\) is the number
of clauses (so \(3n\) is the number of connecting arcs).
Thus the total number of crossings is quadratic in the size of \(\Phi\),
and it is clear that the construction can be done in quadratic time.

\begin{figure*}
\begin{center}
\includegraphics{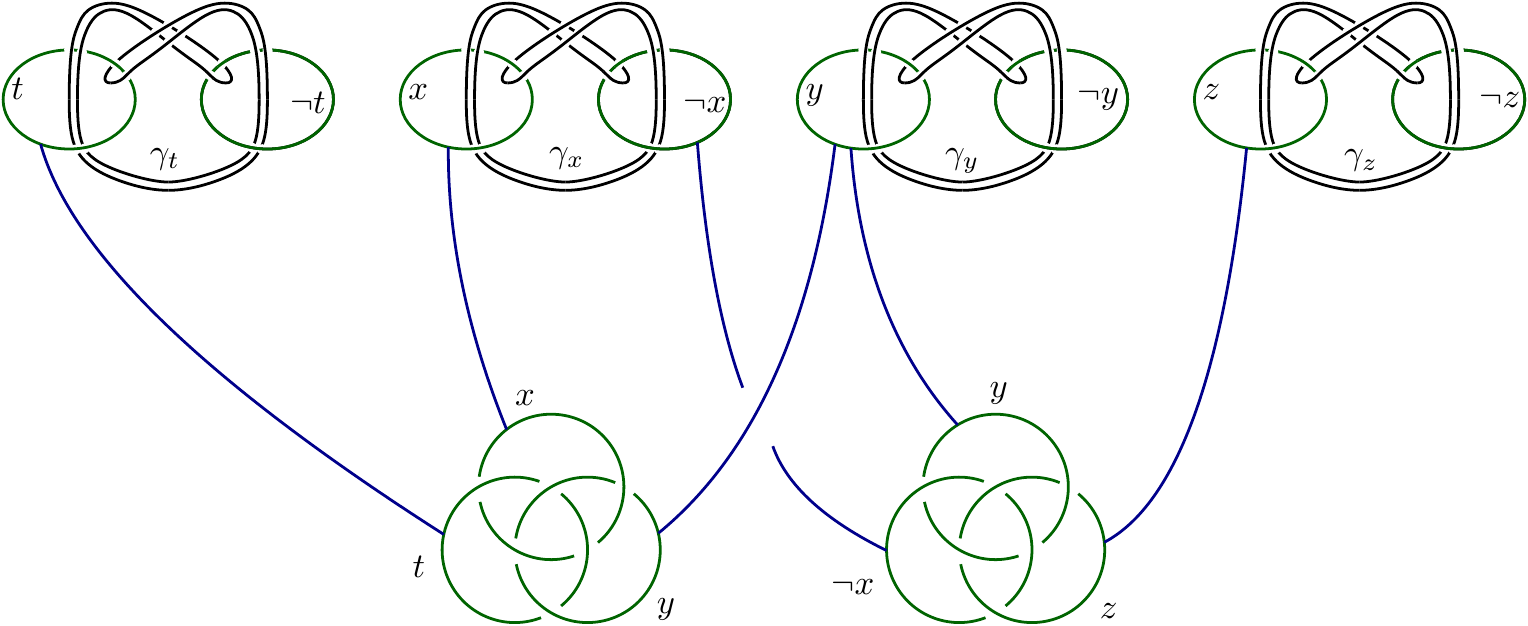}
\caption{The Borromean rings, the variable gadgets and the arcs connecting them
  for the formula $\Phi = (t \vee x \vee y) \wedge
(\neg x \vee y \vee z)$.}
\label{f:construction_full1}
\end{center}
\end{figure*}

\paragraph{Band.} We now modify the diagram by performing a {\em band sum} along each arc, call it $\alpha$,  connecting a variable diagram to a clause diagram:
For each endpoint of $\alpha$ delete a short arc containing that endpoint from
a variable/clause diagram, and then draw two close parallel copies of $\alpha$
that connect the remnants using two disjoint copies of \(\alpha\);
see the two leftmost pictures on 
Figure \ref{f:band_sum_y}. 
Wherever two arcs cross, say \(\alpha\) crossing over \(\alpha'\), we now see four intersections.
In all four intersections we keep the arcs corresponding to \(\alpha\) over the arcs corresponding to \(\alpha'\).
Clearly, the time required for this construction is quadratic in the size of \(\Phi\).

\begin{figure}
\begin{center}
\includegraphics{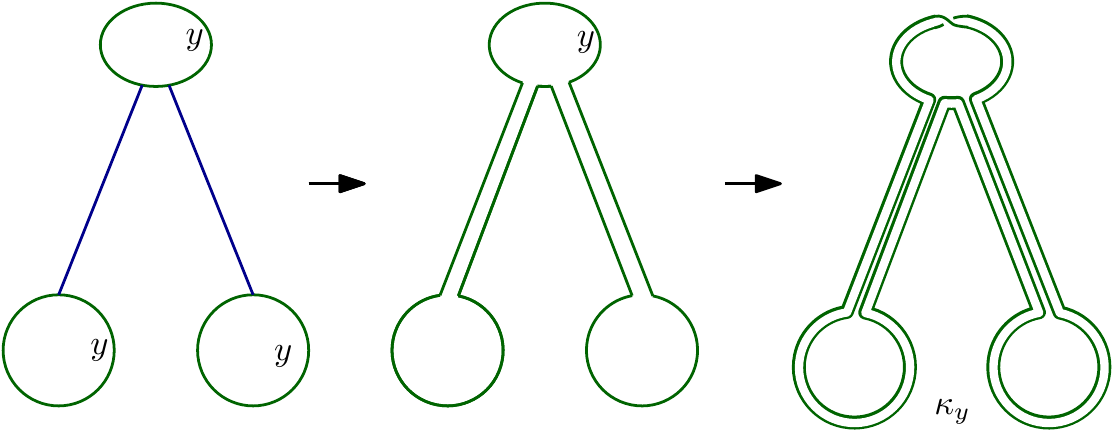}
  \caption{Performing the band sum and then cabling the knot components
  labeled $y$ in the diagram from Figure~\ref{f:construction_full1}.}
\label{f:band_sum_y}
\end{center}
\end{figure}

\paragraph{Cable.}  The final step in our construction of the diagram is \em\((2,1)\)-cabling \em of the components that
correspond to literals (that is, the components that do not correspond to clasps).  This can be described explicitly as follows:
let \(\kappa'\) be a component corresponding to the literal \(\pm x\).  Take two disjoint parallel copies
of \(\kappa'\) and join them together using a single crossing, as shown in the
two rightmost pictures of 
Figure~\ref{f:band_sum_y}. 
Label this component \(\kappa_{\lx}\).  Note that
the time required for this construction, and the number of crossings in the
given diagram, are both quadratic in the size of \(\Phi\).

The 3-manifold corresponding to \(\Phi\) is obtained by Dehn surgery on \(L\).
To complete our construction we determine the surgery coefficients:

\paragraph{The 3-manifold.}  Let $M = M(\Phi) = \mathbb{S}^3(L)$ be the 3-manifold obtained by surgery on the link constructed
above, where the surgery coefficients are \(\emptyset\) on each component that corresponds to a literal and \(3/2\) on each clasp.

\begin{notation}
\label{notation:surfaces}
Before we proceed we summarize the notation that will be used throughout the paper.
Fix a literal \(\lx\).  We will use the following notation for the link
described by the diagram constructed above:
\begin{itemize}
\item[] \(\kappa_{\lx}\): The knot corresponding to \(\lx\) is denoted \(\kappa_{\lx}\).
By construction \(\kappa_{\lx}\) is an unknot.

\item[] $B_\lx$: Note that \(\kappa_{\lx}\) naturally bounds a M\"obius band: it was constructed
from two disjoint parallel curves (that naturally bound an annulus) with a single crossing (that corresponds to
adding a half twist to the annulus).  This M\"obius band is denoted $B_\lx$.
By construction $\kappa_{\lx} = \partial B_\lx$.
\item[]  $\gamma_x$: The {\em clasp} is an unknot that we denote $\gamma_x$.
\item[]  The clasp \(\gamma_{x}\) bounds a disk that we denote \(D_{x}\).
Note that \(D_{x} \cap B_{x}\) consists of exactly one arc (and similarly for \(D_{x} \cap B_{\neg x}\)).

\item[] $L$: The link $L = \bigcup~ (\kappa_x \cup \gamma_x \cup \kappa_{\neg x}$) is the union of the three
components for each variable $x$.  Each clasp $\gamma_x$ is labeled with the surgery coefficient $3/2$, each literal component $\kappa_{\lx}$ is labeled with coefficient $\emptyset$.

\item[] $M=M(\Phi) = \mathbb{S}^3(L)$ is the 3-manifold obtained by surgery on $L$.
\end{itemize}

\end{notation}

\section{$\Phi$ satisfiable $\Rightarrow M$ embeds in $\mathbb{S}^3$} \label{S:onedirection}

First we show that if $M$ is satisfiable, then $M$ embeds into $\mathbb{S}^3$.

\begin{\prop}
\label{proposition:SatisfiableImpliesEmbeds}
If $\Phi$ is satisfiable, then $M = \mathbb{S}^3(L)$ embeds in $\mathbb{S}^3$.
\end{\prop}

\begin{proof}
The idea of the proof is to suitably fill in the `empty tori' in
$\mathbb{S}^3(L)$ corresponding to $\kappa_x$ and $\kappa_{\neg
x}$. That is, it is sufficient to show that we can alter each
coefficient $\emptyset$ to some coefficient in $\Q \cup
\{1/0\}$ so that the resulting 3-manifold is $\mathbb{S}^3$.

Given a satisfying assignment of $\Phi$ we consider every variable $x$
of $\Phi$ and alter the surgery coefficients as follows:
if $x$ is assigned \textsc{True}, we fill $\kappa_x$
with coefficient $1/0$, and if $x$ is assigned \textsc{False} we fill $\kappa_{\neg x}$ with
coefficient ${1}/{0}$. This filling is equivalent to removing $\kappa_x$ or
$\kappa_{\neg x}$ from the link, and from now on we assume that such components have been erased from the diagram.

Since the assignment is satisfying, in each clause $C$ at least one literal is satisfied.
Hence, at least one of the M\"obius bands in the Borromean
rings corresponding to $C$ disappears, and thus the Borromean rings unravel for the other literals involved in $C$. Therefore, we can now use an isotopy
to retract all the bands connecting the clause diagrams to the literals in
the variable gadgets. That is, after this step, we are left with a
link that consists of pairs of components, each pair embedded
in a ball (and the balls are pairwise disjoint), so that each pair
consists of linked unknots, one of which is
the clasp $\gamma_x$, and the other is either $\kappa_x$ (if $x$ is
assigned \textsc{False})
or $\kappa_{\neg x}$ (if $x$ is
assigned \textsc{True}). We now perform a second alteration of the
surgery coefficients, replacing the remaining $\emptyset$ with $3/1$.
We claim that the resulting 3-manifold is $\mathbb{S}^3$: this is
shown using multiple applications of Proposition~\ref{p:simple_twist} as
  pictured in Figure~\ref{f:kirby}.
This shows that the result of the Dehn surgery
on each pair is homeomorphic to the 3-manifold obtained by $- 1/2$
Dehn surgery on the components of an unlink, that is, a link whose components are unlinked unknots.  It is easy to see that this gives $\mathbb{S}^3$.  As an alternative proof that we obtain $\mathbb{S}^3$, we may apply
Example~\ref{e:linked_unknots}(b)
with $k=n=2$ to the penultimate step of Figure~\ref{f:kirby}.

\begin{figure*}
\begin{center}

\def\svgwidth{\textwidth}
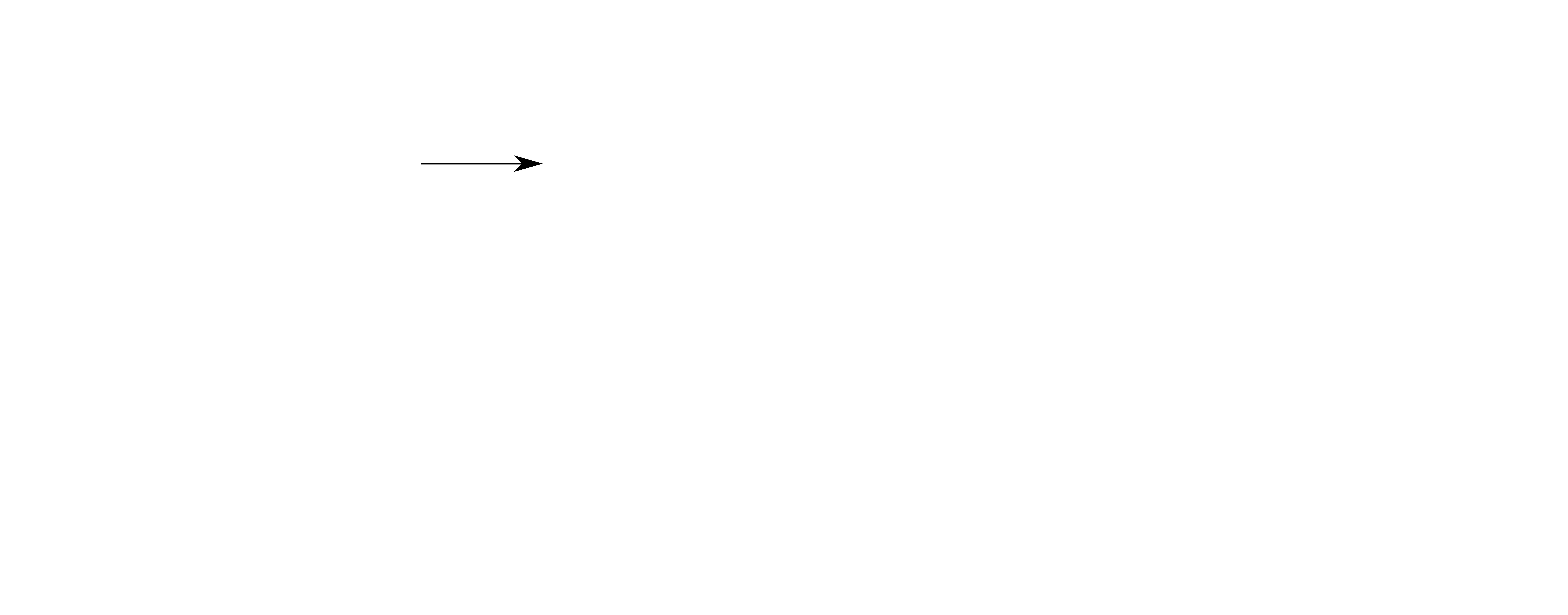

\caption{The resulting components of the link with surgery coefficients
  $L(\Phi)$ after filling the empty tori, and the moves showing that the
  resulting 3-manifold is $\mathbb{S}^3$. The third, fourth and fifth moves are applications of Proposition~\ref{p:simple_twist} to the bolded knot. Note that in the third move  (rightmost arrow),  Proposition~\ref{p:simple_twist} is used with $t=-2$.}
\label{f:kirby}
\end{center}
\end{figure*}

\end{proof}

\section{$\Phi$ satisfiable $\Leftarrow M$ embeds in $\mathbb{S}^3$} \label{S:otherdirection}

\newcommand{\nx}{{\neg x}}

This section is devoted to the proof of the reverse direction, which is much harder.  While Gordon and Luecke \cite{gordon-luecke} showed that a
boundary irreducible 3-manifold whose boundary consists of a single torus admits at most one embedding into $\mathbb{S}^3$,
manifolds with multiple torus boundary components may admit many, even an infinite number, of distinct embeddings
({\it cf.}~\cite{RieckYamashita}).
The main point of the proof  is that $M$ has no {\em accidental} embeddings into $\mathbb{S}^3$.   That is, any embedding is
the result of, for each variable, performing a $1/0$ filling on either the variable or its negation
(but not both).  Interpreting a
$1/0$ filling as \textsc{True} and any other filling as \textsc{False} then allows us to prove:

\begin{\prop}
\label{proposition:EmbedsImpliesSatisfiable}
If $M = \mathbb{S}^3(L)$ embeds in $\mathbb{S}^3$, $\Phi$ is satisfiable.
\end{\prop}

In an effort to expose the underlying structure of our argument, we split the
proof into two sections. First, in this section, we establish the general setup
for proving Proposition~\ref{proposition:EmbedsImpliesSatisfiable}. We identify
a key technical step (Proposition~\ref{prop:genus2compresses}). Still in this
section, we provide a proof of
Proposition~\ref{proposition:EmbedsImpliesSatisfiable} modulo
Proposition~\ref{prop:genus2compresses}. Then, to conclude the proof, we prove
Proposition~\ref{prop:genus2compresses} in
Section~\ref{sec:VxLxBoundaryIrreducible}.

\renewcommand{\L}{L_{\neq 1/0}}

\bigskip

Assume that $M$ embeds in $\mathbb{S}^3$.   The Fox Re-Embedding Theorem~\cite{fox} says that if $M$ embeds in $\mathbb{S}^3$,
then there is an embedding $M \hookrightarrow \mathbb{S}^3$ so that the complement $\mathbb{S}^3 \setminus M$ is the union of
\em handlebodies\em.  Since $\partial M$ consists of tori, these handlebodies are solid tori,\footnote{A \em handlebody
\em is the neighborhood of a graph; in particular, a handlebody whose boundary is a torus is necessarily a solid torus.}
and thus this embedding
is the result of performing a Dehn filling on each component of $\partial M$.

\def\lemb{\ensuremath{L'}}

The Dehn filling of \(\mathbb{S}^3(L)\) that results in \(\mathbb{S}^{3}\) defines a slope on each boundary component that corresponds to a literal \(\lx\).
Let \lemb\ be the link obtained from \(L\) by replacing the surgery coefficients \(\emptyset\)
on \(\kappa_{\lx}\) with the appropriate slope
and, on each clasp $\gamma_x$, retaining the surgery coefficient $3/2$.
Thus $\mathbb{S}^3(\lemb) \cong \mathbb{S}^3$.

We express $\lemb$ as the disjoint union $\lemb = L_{1/0} \cup \L$, where
  $L_{1/0}$
are the components with coefficient $1/0$ and $\L$ are the components with coefficients that are not
$1/0$.  Note then that $\mathbb{S}^3(\L)$ is also homeomorphic to $\mathbb{S}^3$ (erase $L_{1/0}$).

We will use Notation~\ref{notation:surfaces} for the remainder of the section.

\begin{claim}
\label{clm:notBoth10}
For each variable $x$, $\gamma_x \in \L$ and at least one of
$\kappa_x,\kappa_{\neg x}$ is in $\L$.
\end{claim}

\begin{proof}
By construction, for each variable $x$, the clasp $\gamma_x$ has coefficient $3/2$, so $\gamma_x \in \L$.

To complete the proof of the claim assume, for a contradiction, that for some variable \(x\) we have
$\kappa_x,\kappa_{\neg x} \in L_{1/0}$; we may therefore
remove both from \lemb, obtaining \(L''\) so that \(\mathbb{S}^3(L'') \cong \mathbb{S}^{3}\).
However, by our construction, $\gamma_x$ is now separated from all other link
    components of $L''$ (by a $2$-sphere). Example~\ref{example:lensspace} now
    implies that $H_{1}(\mathbb{S}^3(L''))$ has a $\mathbb{Z}_3$-summand which contradicts the fact that $\mathbb{S}^3(L'')$ is a sphere.
\end{proof}

For each literal $\lx$, the component $\kappa_\lx$ was constructed as a
$(2,1)$-cable and it follows that $\kappa_\lx$ is the boundary of a  M\"obius
band $B_\lx$ that is disjoint from all other components of the link.  The clasp
$\gamma_x$ bounds a disk $D_x$ which meets  the M\"obius bands $B_x$ and
$B_{\neg x}$ each in a single arc.
By construction,  for any variable $y \neq x$ we have that
$(B_{x} \cup D_x \cup B_{\neg x}) \cap (B_{y} \cup D_y \cup B_{\neg y}) = \emptyset$. Define the {\em variable link}  to be
$L_x := (\kappa_x \cup \gamma_x \cup \kappa_\nx) \cap \L$, that is, $L_x$ are the
components corresponding to the variable $x$ that have non-$1/0$ surgery coefficients.  Let
$V_x =  N(B_x^* \cup D_x \cup B_\nx^*)$, where $*$ means that we omit $B_\lx$
from the union if $\kappa_\lx \subset L_{1/0}$.  Then, by
Claim~\ref{clm:notBoth10}, $V_x \subset \mathbb{S}^3$ is a genus 1 or 2 handlebody containing the link $L_x$; see Figure \ref{fig:Vx}.

\begin{figure*}[h]
  \center{\includegraphics{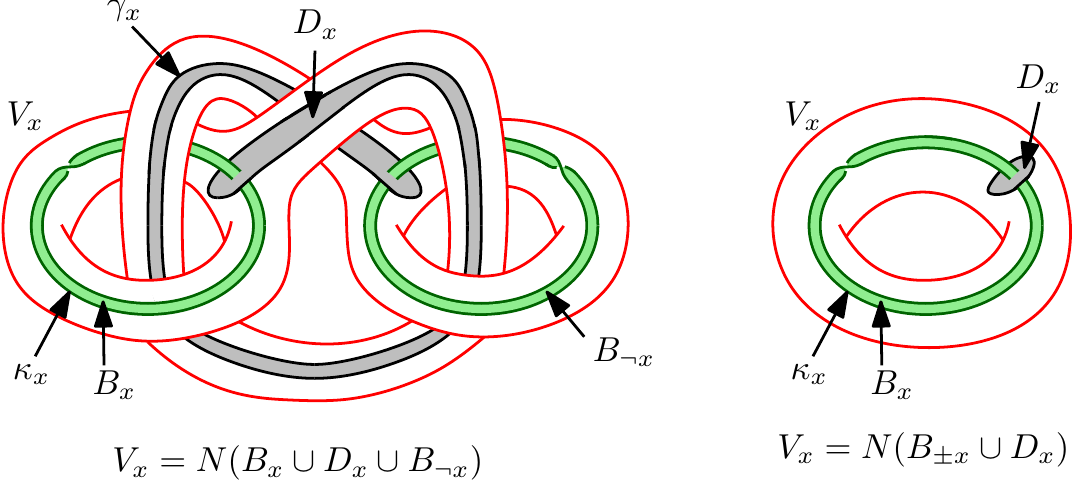}}
\caption{$V_x = N(B_x^* \cup D_x \cup B_\nx^*)$, where $*$ means that we omit $B_\lx$
from the union if $\kappa_\lx \subset L_{1/0}$, either is a handlebody with
  either genus 2 or genus 1. (The picture is given after a homeomorphism
  simplifying $\kappa_{\pm x}$ and for the right picture also $\gamma_x$.)}
\label{fig:Vx}
\end{figure*}

The technical crux of our proof lies in establishing the following proposition:

\begin{restatable}{\prop}{proptechnical}
\label{prop:genus2compresses}
$V_x(L_x)$ is boundary irreducible.
\end{restatable}

For now, we assume it and postpone its proof to Section~\ref{sec:VxLxBoundaryIrreducible}. By construction,  for any variable $y \neq x$ we have that $V_{x} \cap V_{y} = \emptyset$. Let $V = \bigcup V_x$ be the union of these handlebodies over all variables $x$ and $W$ the closure of the complement, $W = \overline{\mathbb{S}^3 \setminus V}$. Of course, every component of $\partial W$ has genus 1 or 2.

\begin{figure}[h]
\center{\includegraphics{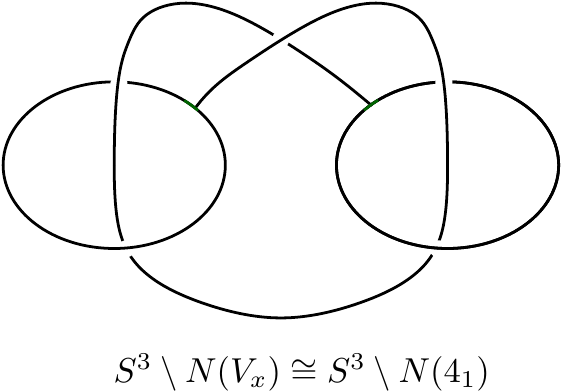}}
\caption{The handcuff graph $4_1$.   The exterior of $V_x$ in $\mathbb{S}^3$ is homeomorphic to the exterior of $4_1$ in $\mathbb{S}^3$.}
\label{fig:handcuff}
\end{figure}

\begin{claim}
\label{clm:genusTwoIncompressibleOutside}
Every genus two component of $\partial W$ is incompressible in $W$.
\end{claim}

\begin{proof}
Any genus two component of $\partial W$ is a boundary $\partial V_x$ for some
  variable $x$ for which both $\kappa_x, \kappa_\nx \subset \L$. Since every variable occurs at most once in each clause, the
  variable link $L_x$ has been connected to and band summed with at most one
  knot of each Borromean ring. Thus this operation did not change its isotopy
  class, and the closure of the complement $\overline{\mathbb{S}^3 \setminus
  V_x}$ is homeomorphic to the exterior of the {\em handcuff graph} $4_1$; see
  Figure~\ref{fig:handcuff}, which was shown by Ishii et al.~\cite{i-tgthku-12}
  to be \emph{irreducible} (for a \emph{different} notion of irreducibility
  than the one we defined in the preliminaries). It is known that a knotted
  handlebody of genus $2$ is irreducible if and only if its exterior has
  incompressible boundary~\cite{i-khdshk-15}, but the proof is hard to extract
  from the multiple references therein, so for completeness we provide another one, tailored to our case, that $\partial V_x$ is incompressible in $\overline{\mathbb{S}^3 \setminus V_x}$.

Let us denote by $H$ the handcuff graph $4_1$. Tsukui~\cite[Example~1]{t-s3s-70} proves that $\pi_1(\mathrm{E}(H))$ is indecomposable with respect to free products,
where here $\mathrm{E}(H) := \overline{\mathbb{S}^{3} \setminus N(H)}$ is the
  exterior of $H$. Suppose ad absurdum that $\partial N(H)$ compresses outside.
  If the compressing disk $D$ separates $\partial N(H)$ (so two tori are obtained)
  then the van Kampen theorem shows that there is a free product decomposition
  unless one of the two sides is simply connected, but every torus in
  $\mathbb{S}^3$ separates $\mathbb{S}^3$ into two components, and neither is
  simply connected.  If $D$ is not separating then banding $D$ to itself we
  obtain a separating compressing disk; see Figure~\ref{f:banding}
  and the explanation in its caption.

\begin{figure*}
\begin{center}
\includegraphics{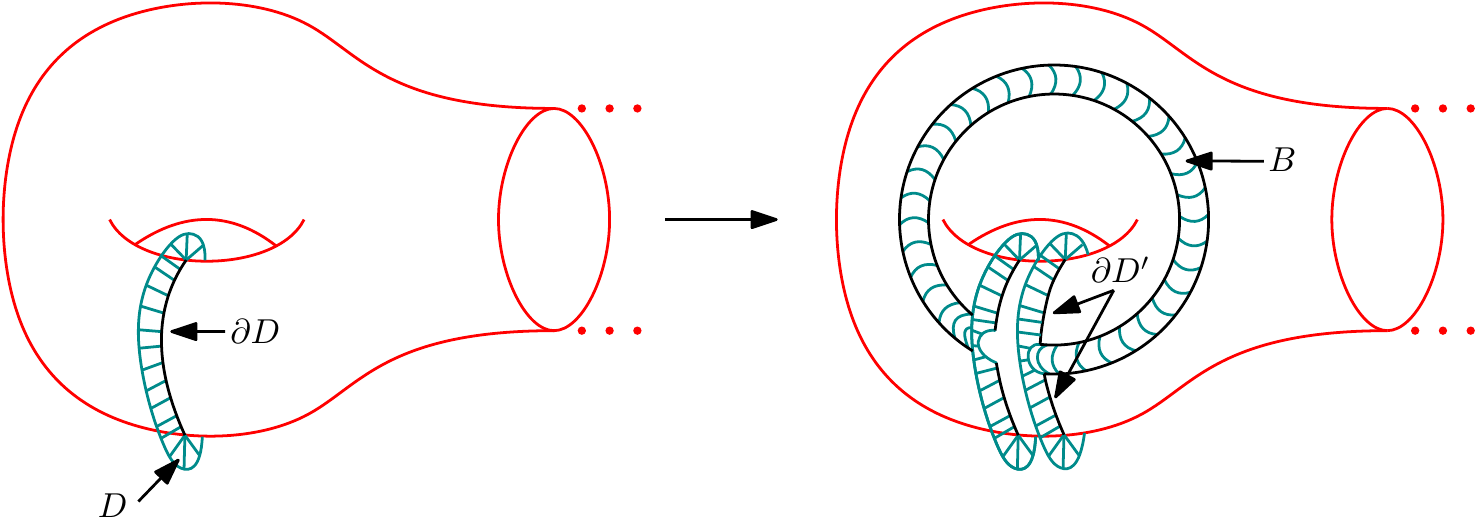}
\caption{Banding the disk $D$ to itself: We consider a disk $D'$ parallel to
  $D$ and a band $B$ connecting the two disks parallel to $\partial N(H)$.
  After interconnecting $D$, $D'$ and $B$ as on the right picture, we obtain a
  separating compressing disk. In this figure, disks are represented by collars of their boundaries.}
  \label{f:banding}
\end{center}
\end{figure*}

Thus, every boundary $\partial V_x$ is incompressible in $\overline{\mathbb{S}^3 \setminus V_x}$.  Since $W \subset \overline{\mathbb{S}^3 \setminus V_x}$, $\partial V_x$ is also incompressible in $W$.
\end{proof}

Therefore any compressible component of $\partial W$ is a torus (we remark that such components may exist).  Let $\Delta \subset W$ be collection of compressing disks, one for each compressible torus.

We claim that we may take the disks in \(\Delta\) to be \em disjointly \em
embedded.\footnote{\label{footnote:innermost}This is an example of an ``innermost disk argument''; see, for example,~\cite[Example~3.9.1]{s-it3m-14}.}
To see this,  we first assume that the disks of \(\Delta\) intersect each other transversely.
Order the disks of \(\Delta\) as \(D_{1},\dots,D_{n}\) and assume that for some
\(i \geq 1\) we have that the \(D_{1},\dots,D_{i}\) are disjointly embedded
(note that this holds for \(i=1\)).  Therefore transversality implies that the
intersection of \(D_{i+1}\) with \(\cup_{j=1}^{i} D_{j}\) is an embedded
1-manifold, and an easy Euler characteristic argument shows that some component
of \(D_{i+1} \setminus (\cup_{j=1}^{i} D_{j})\) is a disk\footnote{This is the \em innermost \em disk.}
 (whose boundary is in \(D_{i+1} \cap D_{j_{0}}\) for some \(1 \leq j_{0} \leq i\)).  We may use the disk to cut and paste \(D_{j_{0}}\) and obtain a new disk that we will use to replace \(D_{j_{0}}\) in \(\Delta\) (without renaming).  The new collection of disks has the exact same boundary, the first \(i\) disks are disjointly embedded, and \(D_{i+1}\) intersects them fewer times than before.  Continuing this way we get a collection \(\Delta\) where the first \(i+1\) disks are disjointly embedded, and the claim follows by induction.

Let $W' := \overline{W - N(\Delta)}$ and $V' := V \cup N(\Delta) = \overline{\mathbb{S}^3 \setminus W'}$.   Thus, $W'$ is
obtained by  cutting $W$ open along the disks $\Delta$ and $V'$ is obtained by
attaching 2-handles to $V$.  Then each link $L_x$ lies in a component of $V'$,
denote it $V_x'$, and note that $\mathbb{S}^3 = V' \cup W'$ where $V' =
\bigsqcup  V_x'$.
Each $V_x'$ is either a genus two handlebody, a
solid torus, or a ball.  Soon we will see that they are all balls.

As $\mathbb{S}^3(\L)$ is homeomorphic to $\mathbb{S}^3$ we can rewrite according to the $\mathbb{S}^3 = W' \cup V'$ decomposition:

 $$ \mathbb{S}^3 \cong W' \cup \Big(\bigcup V_x'(L_x) \Big).$$

\begin{claim}
\label{claim:Wincompressible}
Every non-sphere component of $\partial W'$ is incompressible in $W'$.

\end{claim}
\begin{proof}
By Claim~\ref{clm:genusTwoIncompressibleOutside}, every genus 2 component of $\partial W'$ is incompressible in $W$, and by construction of $\Delta$ every genus 1 component of $\partial W'$ is incompressible in $W$.
\end{proof}

\begin{claim}  For each variable $x$, the 3-manifold $V_x'(L_x)$ is either a ball or is boundary irreducible.
\label{clm:bound_i}
\end{claim}
\begin{proof}
The 3-manifold $V_x'(L_x)$ is embedded in $\mathbb{S}^3$, so if its boundary is
a sphere, then it follows that $V_x'(L_x)$ is a ball.   Thus, it suffices to
show the result for a component  $V_x'(L_x)$ with boundary of positive genus.
And, in that case we have that $V_x'(L_x) = V_x(L_x)$ because 2-handles were
only attached to components with genus 1 boundary, which then become spheres.
Now the claim follows from Proposition \ref{prop:genus2compresses} claiming that $V_x(L_x)$ is boundary irreducible.
\end{proof}

\begin{claim}
\label{clm:balls}
For each variable $x$, the handlebody $V_x'$ is a ball.
\end{claim}
\begin{proof}

Let $S_x = \partial V_x'(L_x)$ be the boundary of $V_x'(L_x)$ and $S = \bigcup S_x$ the union of all these boundaries.  Of course, $S$ is embedded in $\mathbb{S}^3(L_{\neq 1/0})$, which is homeomorphic to $\mathbb{S}^3$.  It then follows that every component of $S$ is either a sphere, or there is a compressing disk $D$ for $S$, i.e., a disk $D$ for which $D \cap S = \partial D$ is an essential curve in $S$. Such a curve must lie in a component of positive genus. But a compressing disk $D$ is impossible, because then either $D \subset W'$, contradicting Claim \ref{claim:Wincompressible}, or $D \subset V_x'(L_x)$ for some $x$, contradicting Claim \ref{clm:bound_i}.

This shows that \(S\) is incompressible; we claim that more is true: \(S_{x}\) is incompressible for every \(x\).   This is again an innermost disk argument, similar to (and in fact simpler than) the argument used above to show that \(\Delta\) may be taken to be embedded.  Suppose, for a contradiction, that some \(S_{x}\) compresses and let \(D\) be a compressing disk for \(S_{x}\)
(\(D\) is not necessarily a compressing disk for \(S\) since it may intersect other components of \(S\)).  We assume that \(D\) intersects \(S\) transversally and minimizes \(\#(D \cap S)\) among all such disks.  An easy Euler characteristic argument shows that some component of \(D\) cut open along \(S\) is a disk; this is the \em innermost \em disk (of course, if \(D \cap S = \emptyset\) the innermost disk is \(D\) itself).  The minimality assumption implies that the boundary of the innermost disk is essential in \(S\), and thus the innermost disk is a compressing disk for \(S\), which we showed above cannot exist.  This contradiction shows that \(S_{x}\) is incompressible for every \(x\).
\end{proof}

This observation allows us to complete the proof of Proposition~\ref{proposition:EmbedsImpliesSatisfiable}:

\begin{proof}[Proof of Proposition~\ref{proposition:EmbedsImpliesSatisfiable}] For each literal $\lx$, assign the value $\lx := \textsc{True}$  if $\kappa_\lx
\subset L_{1/0}$ and $\lx := \textsc{False}$ if $\kappa_\lx \subset \L$.  By Claim \ref{clm:notBoth10}, there is no variable $x$ with both $x$ and $\neg x$ set to \textsc{True}.  Furthermore, since for each variable $x$ the handlebody $V_x'$ is a ball, and by Claim \ref{clm:genusTwoIncompressibleOutside} 2-handles were attached only to solid tori, it follows that every $V_x$ was a solid torus, i.e., exactly one of $x$ and $\neg x$ is \textsc{True}.

Suppose that some clause $C = \lx \vee \l y \vee \l z$ of $\Phi$ is not satisfied. By our assumptions on the formula $\Phi$, the literals $\lx$, $\l y$ and $\l z$ are different, and this clause appears only once in $\Phi$. Consider $B_\lx \cup B_{\l y} \cup B_{\l z}$, the union of 3 M\"obius bands that, because they pass through the diagram for the clause $C$, form the Borromean rings and, in particular, are linked.   But, by Claim \ref{clm:balls}, $B_\lx \subset V_x'$ is contained in a ball that is disjoint from both $B_{\l y}$ and $B_{\l z}$.  This contradicts the fact that they are linked.

We conclude that every clause $C$, hence the total formula $\Phi$, is satisfied.

\end{proof}

\section{$V_x(L_x)$ is boundary irreducible.}
\label{sec:VxLxBoundaryIrreducible}

\newcommand{\rs}{r/s}
\newcommand{\pq}{p/q}

This section is the most technical part of the paper, proving Proposition~\ref{prop:genus2compresses}, which as we saw in the proof of Claim~\ref{clm:bound_i},
is an essential step in showing that $\mathbb{S}^3(L') \ncong \mathbb{S}^3$ unless the surgery coefficients of $L'$ yield a satisfying assignment via the rule $\{1/0 \leftrightarrow \textsc{True},~{\neq}1/0 \leftrightarrow \textsc{False}\}$.

\proptechnical*

Most of the section is a sequence of claims, from which the proof of this proposition will follow.

\begin{figure}[h]
\center{\includegraphics{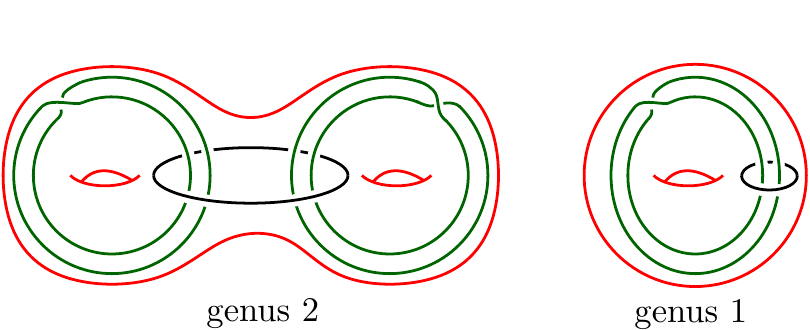}}
\caption{The handlebody $V_x$ with the link $L_x$.}
\label{fig:genus2}
\end{figure}

Recall that  $V_x =  N(B_x^* \cup D_x \cup B_\nx^*)$, where $B_\lx$ is a
M\"obius band bounded by $\kappa_\lx$, $D_x$ is a disk bounded by $\gamma_x$
and  $*$ means that we omit $B_\lx$ from the union if $\kappa_\lx \subset
L_{1/0}$.  Thus $V_x$ is a genus 1 or 2 handlebody that contains a clasp
$\gamma_x$, and 1 or 2 (resp.) literals $\kappa_\lx$; see Figure~\ref{fig:genus2}.

When compared with Figure~\ref{fig:Vx} the handlebody
$V_x$ is drawn here as if it were unknotted. Note that this is purely cosmetic, the difference between
knotted and unknotted is a question of how the handlebody is embedded, whereas
Proposition~\ref{prop:genus2compresses} is a statement about the 3-manifold
itself, irrespective of any embedding. We also observe that we may assume that
$V_x(L_x)$ is unknotted (for arbitrary surgery coefficients on $L_x$). Indeed,
let us consider the switch on the diagram of the variable link depicted at
Figure~\ref{fig:switch}. 
This switch unknots the surrounding handlebody, even
if some Dehn surgery on the link components was already performed.

\begin{figure*}[h]
\center{\includegraphics{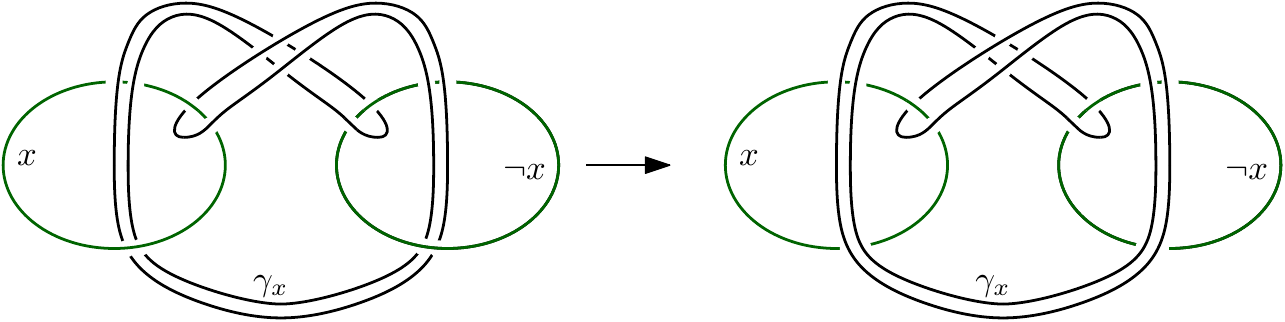}}
  \caption{A switch on the variable diagram that unknots the surrounding
  handlebody.}
\label{fig:switch}
\end{figure*}

We adjust our notation to suppress $x$ and to emphasize slopes over link
components. If $V_x$ is a genus 2 handlebody, we write $V(r_+, s,
r_-)$ for the 3-manifold obtained after surgery on $V_x$ where $r_\pm$ is the
surgery coefficient of $\kappa_\lx$ and $s$ is the surgery coefficient on
the clasp $\gamma_x$. Similarly, if $V_x$ is a solid torus, we write $V(r_\pm,
s)$ for the 3-manifold obtained after surgery on $V_x$ with $r_\pm$ on
$\kappa_\lx$ and $s$ on $\gamma_x$. As usual, a coefficient of
$\emptyset$ means that the component has been drilled out but not filled.  In
this notation  $V_x(L_x)$ is homeomorphic to either $V(r_+, 3/2,
r_-)$ or $V(r_\lx,3/2)$,  where each $r_\pm \neq 1/0$.

If $r$ and $r'$ are slopes in a torus, let $\Delta(r,r')$ denote their {\em
distance}, that is the minimum number of intersections taken over all pairs of
curves that have slopes $r$ and $r'$, respectively. If $r= t/u$
and $r' = v/w$ with respect to some homology basis for the torus, then the
distance is easily computed, $\Delta(r,r') = \Delta(t/u, v/w) = |tw -
uv|$.  Define the distance between a filling and a non-filling to be
$\Delta(r,\emptyset)= \infty$.

\begin{figure*}[h]
\begin{center}
  \includegraphics{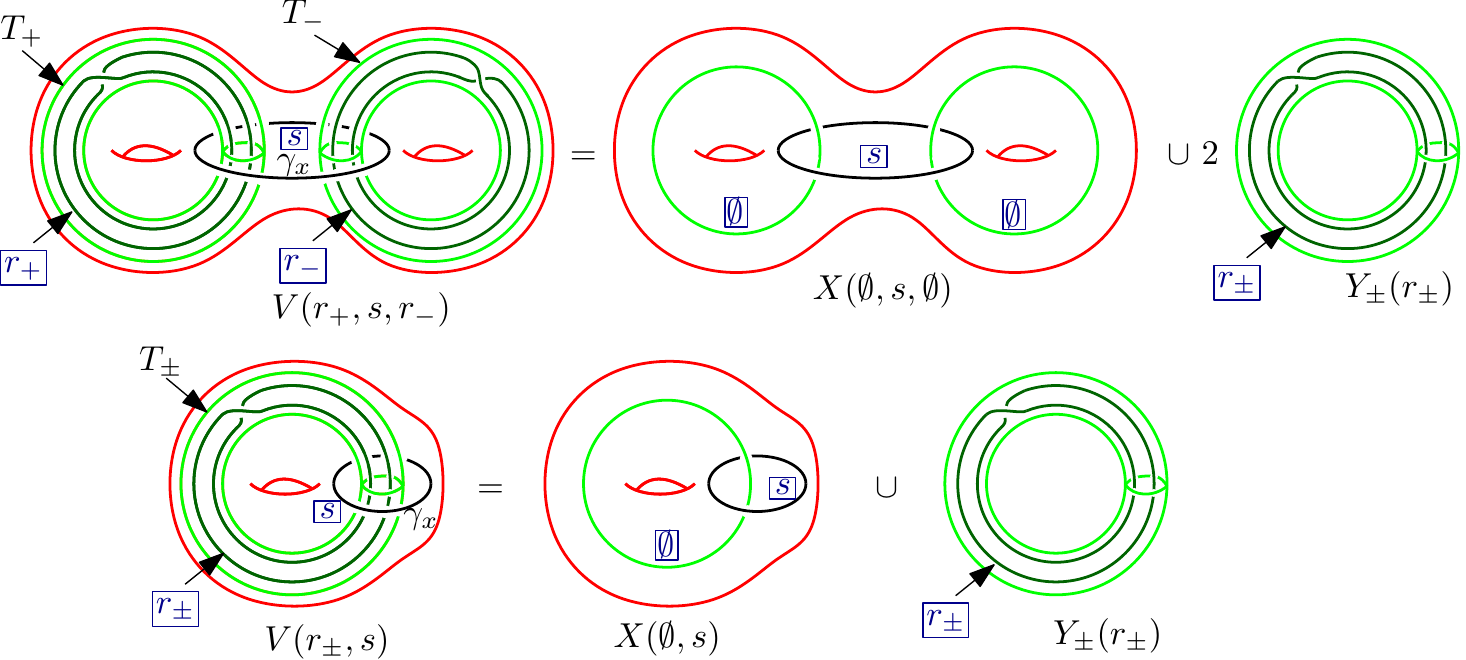}
\caption{$V$ can be cut along tori $T_+$ and $T_-$ to obtain $X$ and one or two copies of $Y_\pm$. }
\label{fig:cutAlongTori}
\end{center}
\end{figure*}

\begin{figure}[h]
\begin{center}
  \includegraphics{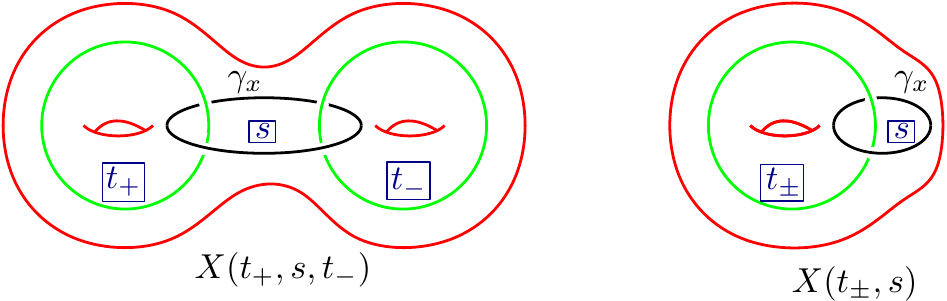}
  \caption{$X(t_+,s,t_-)$ and $X(t_\pm,s)$.}
\label{fig:X}
\end{center}
\end{figure}

Recall that each literal component bounds a M\"obius band $B_\lx$.  The
boundary of a regular neighborhood of the band is a torus $T_\pm = \partial
N(B_\lx) \subset V_x$ that separates each literal knot $\kappa_\lx$ from the clasp and
$\kappa_{\mp x}$ (if present).  Cut $V(r_+, s, r_-)$ along these tori as shown in Figure
\ref{fig:cutAlongTori}. This expresses $V(r_+, s, r_-)$ as a union $V(r_+, s,
r_-)= X(\emptyset, s, \emptyset) \cup Y_+(r_+) \cup Y_-(r_-)$
or $V(r_\pm, s) = X(\emptyset, s) \cup Y_\pm(r_\pm)$; see
Figure~\ref{fig:cutAlongTori}. Here $X(t_+, s, t_-)$ or $X(t_\pm, s)$ is the
manifold obtained by Dehn surgery on the link inside $V_x$ on
Figure~\ref{fig:X} (we keep the clasp $\gamma_x$ but the other knot(s) are
different), and $Y(r_\pm)$ is depicted at Figure~\ref{fig:cutAlongTori}.

\medskip

We first introduce the following two standard lemmas, whose proofs can be found in Schultens~\cite[Section~3.9]{s-it3m-14}.

\begin{\lemm}
\label{lem:compDiskDisjoint}
Suppose that $F \subset M$ is a properly embedded essential surface in a 3-manifold $M$.  If $M$ is reducible then there is a reducing sphere $S$ for $M$ so that $S \cap F = \emptyset$.   If some  component of $\partial M$ has a compressing disk in $M$, then there is a compressing disk $D$ for that boundary component for which  $D \cap F = \emptyset$.
\end{\lemm}

\begin{\lemm}
\label{lem:isotopeEssential}
Let $F$ and $F'$ be properly embedded essential surfaces in an irreducible and boundary irreducible 3-manifold. Then $F$ and $F'$ can be isotoped so that $F \cap F'$ is {\em essential}, that is  each component of the intersection is a curve (loop or arc) that is essential in both $F$ and $F'$.
\end{\lemm}

Then, our first step is to prove that $X(\emptyset, \emptyset)$ and
$X(\emptyset, \emptyset, \emptyset)$ are irreducible and boundary irreducible.

\begin{claim}
\label{lem:SIncompressibleInV}
  $X(\emptyset, \emptyset)$ and $X(\emptyset, \emptyset, \emptyset)$ are irreducible and boundary irreducible.
\end{claim}
\begin{proof}

  First, we observe that $X(\emptyset, \emptyset)$ is homeomorphic to the product $\{\textit{pair of pants}\} \times
  \mathbb{S}^1$ (see Figure~\ref{fig:X}) which is both irreducible and boundary irreducible (see for example Jaco~\cite[Chapter VI]{jacoBook}).

Next, consider the case of $X(\emptyset,\emptyset,\emptyset)$.  It is the exterior of a three component chain embedded in the genus two handlebody as indicated in Figure \ref{fig:cutAlongTori}.  When the handlebody is
  standardly embedded in $\mathbb{S}^3$, as pictured, the chain is a non-trivial link in
  $\mathbb{S}^3$.  It follows that the three inner torus boundary components are
  incompressible in the exterior of that link in $\mathbb{S}^3$, hence they are also
  incompressible in $X(\emptyset,\emptyset,\emptyset)$.

There remains to show that $X(\emptyset,\emptyset,\emptyset)$ is irreducible and that its genus two boundary component $F$ is incompressible in $X(\emptyset,\emptyset,\emptyset)$. For this we will rely on Lemma~\ref{lem:compDiskDisjoint}.

First, recall that the central clasp $\gamma_x$ bounds a disk $D_x \subset V_x$. Then
  $P = D_x \cap X(\emptyset,\emptyset,\emptyset) \subset X(\emptyset,\emptyset,\emptyset)$ is a properly embedded pair of pants; see Figure
  \ref{fig:genus2Surfaces}. Moreover, $P$ is essential:  Compressing  $P$ would
  yield two surfaces, an annulus and a disk, each with essential boundary.  But
  the disk would be a compression for one of the inner torus boundary
  components, contradicting our previous observation that they are
  incompressible in $X(\emptyset,\emptyset,\emptyset)$.

Now, let us assume that there is a reducing sphere for
$X(\emptyset,\emptyset,\emptyset)$ or a compressing disk for $F$. By Lemma
\ref{lem:compDiskDisjoint}, there is a reducing sphere or compressing disk
disjoint from $P$.  But, then that sphere or disk  is a reducing sphere or
compressing disk for $X(\emptyset,\emptyset,\emptyset) \setminus N(P)$.
However, $X(\emptyset,\emptyset,\emptyset) \setminus N(P)$ is homeomorphic to the product $F \times [0,1]$, which is irreducible and boundary irreducible, a contradiction.
\end{proof}

Let $A$ be a properly embedded annulus in an irreducible 3-manifold with incompressible boundary.  Then $A$ is incompressible, for the alternative implies that the 3-manifold's boundary is compressible.   We say that $A$ is {\em peripheral} if it is boundary compressible, {\em spanning} if it meets two distinct boundary components, and {\em cabling} if it is not peripheral and meets only one boundary component.

The following lemma applies well known results on Dehn filling~\cite{cgls,scharlemannDiskSphere}.

\begin{lemma}
\label{lem:IncompressibleIrreducibleNotAProduct}
Let $M$ be an irreducible and boundary irreducible 3-manifold that is not
  homeomorphic to $\{torus\} \times [0,1]$.  Let $A \subset M$ be a spanning
  annulus whose boundary has slope $t$ in a torus boundary component $T \subset
  \partial M$. If $\Delta(s,t) > 1$, then $M(s)$, the 3-manifold obtained by
  performing a Dehn filling on $T$ with slope $s$, is irreducible, boundary irreducible and is not homeomorphic to $\{torus\} \times [0,1]$.
\end{lemma}

\begin{proof}
Let $S$ be the other boundary component met by the spanning annulus $A$.   Note
  that the surface $S$ is compressible in $M(t)$ (and possibly for slopes that
  meet $t$ once). But since $M$ is not homeomorphic to $ \{ torus \} \times
  [0,1]$ and $\Delta(s,t) > 1$ we can apply a theorem of Culler, Gordon,
  Luecke and Shalen~\cite[Theorem 2.4.3]{cgls} to conclude that $S$  is
  incompressible in $M(s)$.   Similarly, Theorem 2.4.5 of the same
  article~\cite{cgls} implies that when $\Delta(s,t) > 1$ every other boundary
  component $S'$ is incompressible in $M(s)$.

Next we show that $M(s)$ is irreducible.   Note first that any cabling
  annulus $A'$ for $M$ meeting $T$ has slope $t$.  For Lemma \ref{lem:isotopeEssential}
  implies that $A$ and $A'$ can be isotoped to intersect essentially.   Because
  the boundary curves of $A'$ are in $T$, any arc component $A \cap A'$ has
  both endpoints in $T$. But then, such an arc is inessential in $A$.  Thus
  there are no intersection arcs and both annuli have slope $t$.  We are now in
  a position to apply a theorem of Scharlemann~\cite{scharlemannDiskSphere}
  (with $M = M(t), M' = M(s))$, which shows that when  $\partial M(t)$  is
  compressible then either $M(s)$ is a solid torus, $s$ is the slope
  of a cabling annulus for $T$, or $M(s)$ is irreducible.    But, we have
  proved that $M(s)$ has incompressible boundary, so it is not a solid
  torus, and the only possible slope for a cabling annulus is $t$, and thus
  since $\Delta(s,t) > 1$, $s$ is not a cabling annulus.   It follows that $M(s)$ is irreducible as claimed.

Finally assume, by way of contradiction, that $M(s)$ is  $\{torus\} \times
  [0,1]$.    This implies that $M$ has three torus boundary components $S,T$
  and another torus, call it $S'$.   Let $r$ be the slope of $A$ on $S$.
  We may write $M(s) = \{torus\} \times [0,1] = A_r \times \mathbb{S}^1$, a
  product where $A_r$ is an annulus with slope $r$ (actually for any
  slope).  Perform a second Dehn filling with slope $r$ on $S$.  The
  attached solid torus has a product structure $D_r \times \mathbb{S}^1$, where
  $D_r$ has slope $r$ in $T$, and so after the attachment we have
  that the filled 3-manifold is a solid torus with product structure $Z =
  (D_r \cup A_r) \times \mathbb{S}^1$. Then perform a third filling on $S'$
  along any slope with meets the meridional disk $D_r \cup A_r$ once.
  This produces  $\mathbb{S}^3$.  Now consider the torus $T \subset \mathbb{S}^3$.   On the
  attached solid torus side, there is a compressing disk for $T$ with slope
  $s$.  But, $A \cup D$, where $D$ is a meridional disk for the torus
  attached to $S$, is a compressing disk for the other side with slope $t$.
  Thus we have expressed $\mathbb{S}^3$ as a lens space, the union of two solid tori
  glued along the torus $T$.   As discussed in Example~\ref{example:lensspace},
  the homology of this 3-manifold is $\mathbb Z_q$, where $q$ is the intersection
  number in $T$ of the boundaries, $s$ and $t$, of the meridional disks.
  Since $\mathbb{S}^3$ is a homology sphere we have  $q = \Delta(s,t)=1$, contradicting
  our assumption that this quantity is at least 2. \end{proof}

Repeated application of the above lemma yields the following conclusion.  We
emphasize that it covers the possibility that  $t_\pm = \emptyset$: just don't perform all the fillings.

\begin{claim}
\label{lem:FCompressesInXT} If $\Delta(t_\pm,1/0) > 1$, then $X(t_\pm,3/2)$ is boundary irreducible.  If $\Delta(t_+,1/0)>1$ and $\Delta(t_-,1/0)>1$, then $X(t_+,3/2,t_-)$ is boundary irreducible.
\end{claim}

\begin{figure}[h]
\center{
  \includegraphics{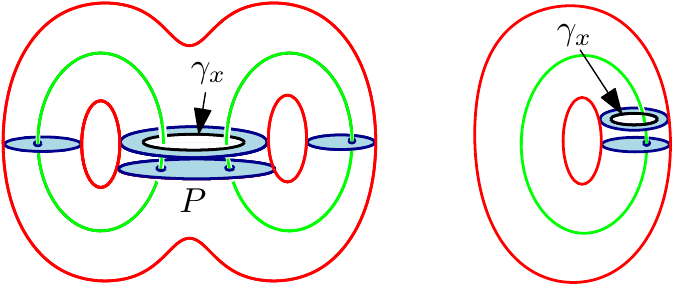}
}
  \center{
}
\caption{Annuli in $X$, and a pair of pants $P$ in $X(\emptyset,\emptyset,\emptyset)$.}
\label{fig:genus2Surfaces}
\end{figure}

\begin{proof}
  We argue both cases at once. Observe that both $X(\emptyset, \emptyset)$ and
  $X(\emptyset, \emptyset, \emptyset)$ have three torus boundary components
  each, so neither is
homeomorphic to $\{torus\} \times [0,1]$.
Claim \ref{lem:SIncompressibleInV}
shows that both types are irreducible and boundary irreducible, so suitable for
application of Lemma~\ref{lem:IncompressibleIrreducibleNotAProduct}.  In both
cases, there is a spanning annulus meeting the outer boundary component that
meets the clasp component in slope $0/1$; see Figure \ref{fig:genus2Surfaces}.
Moreover, for any inner torus boundary component $T_\pm \subset \partial X$,
there is a spanning annulus from the outer boundary that has slope $1/0$ on
$T_\pm$.   Since, $\Delta(3/2,0/1)=2>1$ and, by assumption, $\Delta(t_\pm,1/0)>1$, we can use
the prior lemma two or three times in (any) sequence to conclude that the
filled 3-manifold is boundary irreducible.

\end{proof}

We now turn our attention to $Y_\pm(r_\pm)$.  Fortunately, $Y_\pm(\emptyset)$
is a {\em cable space}, a Seifert fibered space over the annulus with one
exceptional fiber, and we have a complete understanding of when $T_\pm$, its
outer torus boundary component, is compressible in $Y_\pm(r_\pm)$. This follows
directly from a lemma of Gordon and Litherland~\cite{gordon-litherland} that
classifies essential planar surfaces in a cable space.

\begin{claim}
\label{lem:T+Compresses}
If $T_\pm$ is compressible in $Y_\pm(r_\pm)$ then one of the following holds:
\begin{enumerate}
\item $r_\pm = 2/1$ and $t_\pm = 1/2$, or,
\item $r_\pm = \frac {1+2k}{k}$ and $t_\pm = \frac {1+2k}{4k}$, for some $k \in \mathbb Z$,
\end{enumerate}
where $t_\pm$ is the slope of the curve in $T_\pm$ that bounds a disk in
  $Y_\pm(r_\pm)$.
\end{claim}

\begin{proof}

If the outer boundary component does compress in $Y_\pm(r_\pm)$, then there will be an essential  planar surface in $Y_\pm(\emptyset)$ meeting the outer boundary component $T_\pm$ in a
  single closed curve, let $t_\pm$ denote its slope; and, some number of curves
  on the inner torus boundary component, let $r_\pm$ denote their slope.
  Gordon and Litherland classify essential planar surfaces in Lemma 3.1 of
  \cite{gordon-litherland} (using $p=1, q=2$ to apply their result).  The only essential planar surfaces that meet the outer boundary once and inner boundary at least once, are their cases (3) and (4) which correspond directly to conclusions (1) and (2), respectively.
\end{proof}

Our final claim shows that if $F$ compresses, then there is a compressing disk which respects the decomposition along the tori $T_\pm$.

\begin{claim}
\label{lem:meetsInDisks}
Suppose that $V(r_+,s,r_-)$ or $V(r_\pm,s)$ is boundary reducible.
  Then there is a boundary reducing disk $D$ so that $D \cap Y_+(r_+)$ and $D
  \cap Y_-(r_-)$ are each a (possibly empty) union of compressing disks for $T_+$ and $T_-$, respectively.
\end{claim}

\begin{proof}

Let $D$ be a disk that is transverse to both $T_+$ and $T_-$.   Define the weight of $D$ to be the number of components of intersection with these tori, $wt(D) = |D \cap (T_+ \cup T_-)|$.   If $b$ is a (curve) component of $D \cap (T_+ \cup T_-)$, then $b$ is a loop that bounds a sub-disk $D_b \subset D$.  Define the weight of $b$ to be the weight of the interior of this disk,   i.e. $wt(b) = wt(int(D_b)) = | int(D_b) \cap (T_+ \cup T_-)|$.

Over all compressing disks for $F$ in  $V(r_+,s,r_-)$ that are transverse  to $T_+$ and $T_-$, let $D$ be one that minimizes the weight $wt(D)$.  We will show that $D$ meets $Y_+(r_+)$ and $Y_-(r_-)$ as claimed, only in compressing disks.

First note that every component of $D \cap (T_+ \cup T_-)$ is a curve that is essential in one of the tori,  $T_+$ or $T_-$. For otherwise, among intersection loops that are inessential in, say, $T_+$, we could choose one, $a$,  that is innermost in \(T_{+}\) (recall Footnote~\ref{footnote:innermost}).  That is, $a$ bounds a disk in $T_+$ that is disjoint from $D$.   Then we can modify $D$ by replacing $D_a$,  the disk $a$ bounds in $D$, with the disk $a$ bounds in $T_+$.  But then, a slight isotopy makes the new disk transverse to $T_+$, eliminates the intersection curve $a$, and reduces the weight $wt(D)$  by at least 1, a contradiction.

We now show that for each torus, all intersection curves of $D$ with that torus
have the same weight:  Suppose that there are intersection curves of $D$ with, say, $T_+$ of differing weights.  Then, because the intersection curves are all parallel in $T_+$, they cut $T_+$ into annuli; and one of those annuli, call it $A$, has boundary curves $b$ and $c$ that are intersection curves with different weights, say  $wt(b) < wt(c)$.   Let $D_b$ and $D_c$ be the respective sub-disks of $D$ that they bound.   Form a new disk $D'$ with the same boundary as $D$ by replacing $D_c$ with  (a slight push-off of) $A \cup D_b$.  But this eliminates $wt(c)-wt(b)>0$ intersections (and one more in the case that  $D_b$ and $D_c$ start on opposite sides of $T_+$) and contradicts our assumption that $wt(D)$ was minimized.

Finally, note that constant weights on each torus implies that every component
  of $D \cap Y_+(r_+) \cup  Y_-(r_-)$ must be a (compressing) disk.  Any non-disk component is a planar surface $P \subset D$, whose boundary components all lie in either $T_+$ or $T_-$.  But the weight of its outer boundary component strictly exceeds the weight of each of  its inner boundary components, contradicting the fact that the weights for all curves in that torus are equal.  The conclusion of the lemma follows.

\end{proof}

With the above lemmas, we are able to give the proof of Proposition \ref{prop:genus2compresses}.
\begin{proof}[Proof of Proposition \ref{prop:genus2compresses}]

We prove that $V_x(L_x)$ is boundary irreducible.  Recall that $V_x(L_x) =
  V(r_+,3/2,r_-)$ or $V_x(L_x) = V(r_\pm,3/2)$ where $r_\pm \neq 1/0$.

By way of contradiction, assume that $V_x(L_x)$ is boundary reducible.
  Applying Claim \ref{lem:meetsInDisks}, we may find a compressing disk $D$ for
  $\partial V_x(L_x)$ so that $D \cap (Y_+(r_+) \cup Y_-(r_-))$ is a collection
  of compressing disks.  This collection cannot be empty for this would
  imply that $X(\emptyset,3/2)$ or $X(\emptyset,3/2,\emptyset)$ is boundary
  reducible, contradicting Claim \ref{lem:FCompressesInXT}.  So then $D \cap X$
  is  a {\em punctured disk}, that is a planar surface $P$ with one boundary
  component in $\partial V_x(L_x)$ and all others in $T_+$ or $T_-$.
  (Here $X$ stands for either $X(\emptyset,3/2)$ or
  $X(\emptyset,3/2,\emptyset)$.)

First suppose that $P$ meets only one of the tori, say $T_+$, and  in slope
  $t_+$.   This implies that $\partial V_x(L_x)$ compresses in $X(t_+,3/2)$ or
  $X(t_+,3/2,\emptyset)$ and by Claim \ref{lem:FCompressesInXT} we conclude
  that  $\Delta(t_+,1/0) \leq 1$.  Since $D$ meets $Y_+(r_+)$ in compressing
  disks, we can apply Claim \ref{lem:T+Compresses}.  The only available slope satisfying
  $\Delta(t_+,1/0) \leq 1$ is $t_+ = 1/0$.  But this implies that $r_+ = 1/0$, a contradiction.

Finally, consider the case that $P$ meets both tori, $T_+$ in slope $t_+$ and
  $T_-$ in slope $t_-$.  Again, applying Claim \ref{lem:FCompressesInXT} we
  find that for one slope, say $t_+$, we have $\Delta(t_+,1/0) \leq 1$.  But as
  before, we can also apply Claim \ref{lem:T+Compresses} on $Y_+(r_+)$ to
  conclude that $t_+ = 1/0$ and $r_+ = 1/0$, a contradiction.

\end{proof}


\section{Triangulating $\mathbb{S}^3(L)$}

In this section we show that a triangulation of $M = \mathbb{S}^3(L)$ can be computed efficiently.  The results in this section are known (see, in particular,  Hass, Lagarias and Pippenger~\cite[Lemmas~7.1 and~7.2]{hlp-ccklp-99}).  However, since they are very useful we decided to include a complete discussion, a little more general than is needed for our work.

In Proposition~\ref{P:TriangulatingE(L)} we calculate efficiently a triangulation of a link exterior, so that the meridian of each component embed in the \(1\)-skeleton of the triangulation induced on that boundary.\footnote{By the \em triangulation induced on the boundary \em we mean the restriction of the given triangulation to boundary.}  However, the longitude may not embed in that triangulation.  This forces us to discuss another slope, call the \em blackboard framing\em, which we now define:

\begin{definition}[blackboard framing] Let \(k \subset \mathbb{S}^{3}\) be a
  knot with a given diagram \(D \subset S\) (throughout this section \(S
  \subset \mathbb{S}^{3}\) is a $2$-sphere).  The \em blackboard framing \em of
  \(k\) is the slope on \(\partial N(k)\) represented by the simple closed curve that
  is parallel to \(k\) in \(S\), that is, obtained by pushing \(k\) to one side
  on \(S\);
  see Figure~\ref{f:bf_longitude}.
\begin{figure}[h!]
\begin{center}
\includegraphics{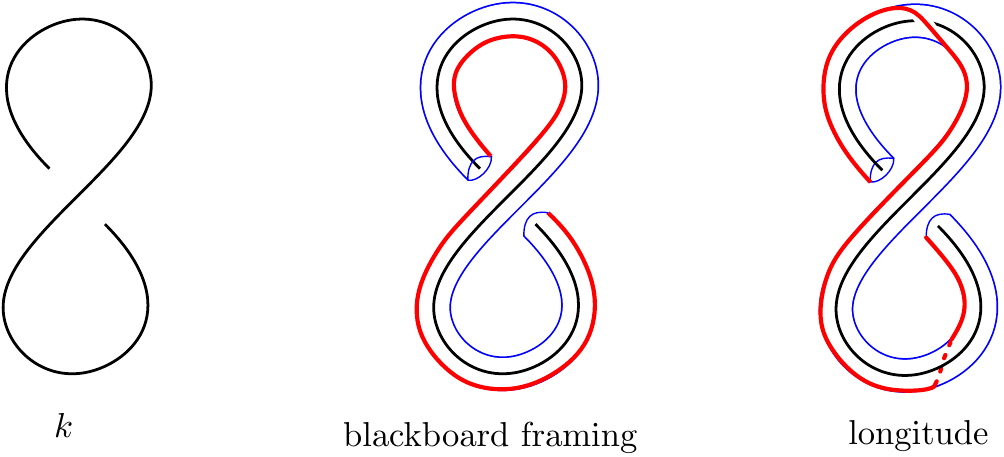}
\end{center}
  \caption{A diagram of a knot $k$ (actually the unknot) and the corresponding
  blackboard framing and the longitude in $\partial N(k)$. The blackboard framing may link with
  $k$ whereas the longitude is never linked with $k$.}
  \label{f:bf_longitude}
\end{figure}
\end{definition}

\begin{remark}
\label{remark:BBframing}
A few comments are necessary before we proceed:
	\begin{enumerate}
	\item The curve defining the blackboard framing intersects a meridional curve once and thus
	  the blackboard framing is an integral slope, say \(n/1\).  It is not, in general, the
	    longitude (whose slope is \(0/1\)).

\item The integer \(n\) above  is called the \em writhe \em of \(k\); see, for example, in Definition~3.4 of~\cite{LickorishBook}.   The writhe is the sum of the signs of the crossings (the sign of a crossing is defined in Page~11 of~\cite{LickorishBook}).  In particular, the absolute value of the writhe cannot exceed the crossing number.  In that sense the blackboard framing is not too far from the longitude \(0/1\).

	\item The blackboard framing \em does \em depend on the diagram \(D\) and its slope changes by \(\pm 1\) under a Reidemeister move of type I (the sign depends on the sign of the crossing introduced/removed; see Figure~\ref{f:bf_longitude}). However, the writhe does not change under Reidemeister moves of type II and III.

	\item  When considering a link we will discuss the blackboard framing of each component while ignoring the remaining components.
	\item A link component \(\gamma\) that has no self crossing has writhe \(0\) (regardless of crossings with other components).  More generally, a link component \(\gamma\) that can be changed to the diagram with no self crossings using only Reidemeister move type II moves has writhe \(0\) (again ignoring other components).  This shows each clasp \(\gamma_{x}\) in our construction has writhe \(0\) (recall Figure~\ref{fig:switch}), and so the blackboard framing there is the longitude \(0/1\).

	\end{enumerate}
\end{remark}

\begin{\prop}\label{P:TriangulatingE(L)}
Let \(D \subset S \subset \mathbb{S}^{3}\) be a link diagram  of an \(n\) component link \(L \subset \mathbb{S}^{3}\) with \(c\) crossings.
Then there exits a triangulation \(\mathcal{T}_{E(L)}\) of \(E(L)\) so that the following conditions are satisfied:
	\begin{enumerate}
	\item \(\mathcal{T}_{E(L)}\) can be calculated in time \(O(c+n)\) and has \(O(c+n)\) tetrahedra.  Moreover, if \(D\) is connected then \(O(c+n)\) can be replaced by \(O(c)\) in both places.
	\item Both the meridian and the blackboard framing embed in the triangulation induced by \(\mathcal{T}_{E(L)}\) on each component of \(\partial E(L)\).
	\end{enumerate}
\end{\prop}

\begin{proof}
By applying Reidemeister moves of type~II to \(D\) we may assume that it is
  connected; this increases the number of crossing to \(c+2(n-1)\) at most.   By
applying a Reidemeister move of type~II to \(D\) at each monogon
we may assume that no complementary region
is a monogon; at worst, this triples the number of crossings.
By Remark~\ref{remark:BBframing}~(3), these moves do not change the writhe of the components of \(L\).
After these changes, the closure of each component of \(S \setminus D \) is an \(n\)-gon with \(n \geq
2\).  We will assume that \(D\) is a link diagram with \(c\) crossing satisfying these conditions and construct a
triangulation is time \(O(c)\) using \(O(c)\) tetrahedra; the proposition
will follow.

Ignoring the crossings, it is convenient to view \(D\)  as a  $4$-valent graph $G$ embedded in $S$ with $c$ vertices and \(2c\) edges. Note that the edges of \(G\) are the arcs obtained by cutting \(D\) (or \(L\)) at the crossings.

We now construct \(\mathcal{T}\), a triangulation of \(\mathbb{S}^{3}\) so that \(L\) embeds in the \(1\)-skeleton of \(\mathcal{T}\)
\begin{figure}
\begin{center}
\includegraphics[width=9cm]{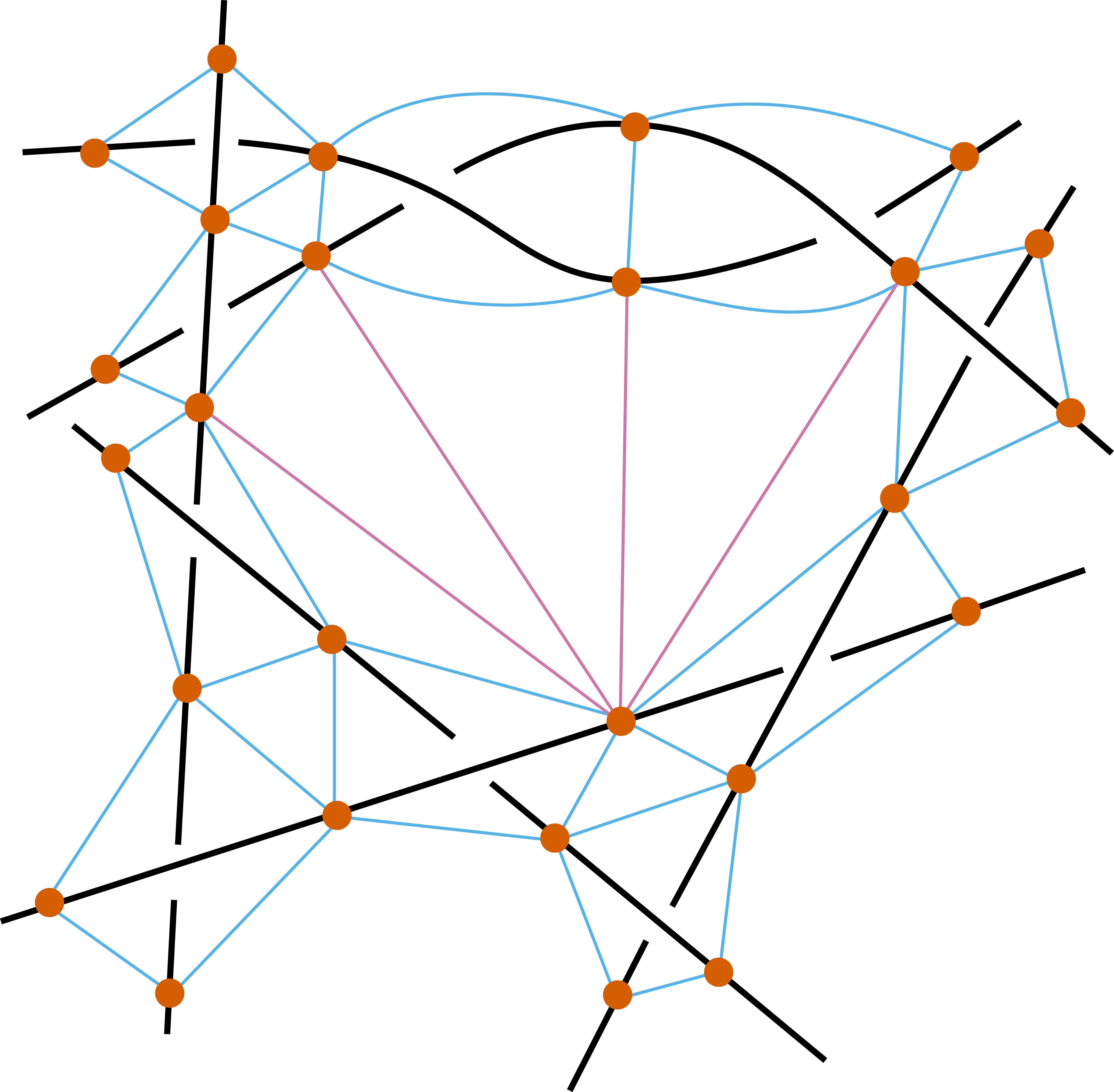}\hspace{1cm} \includegraphics[width=4cm]{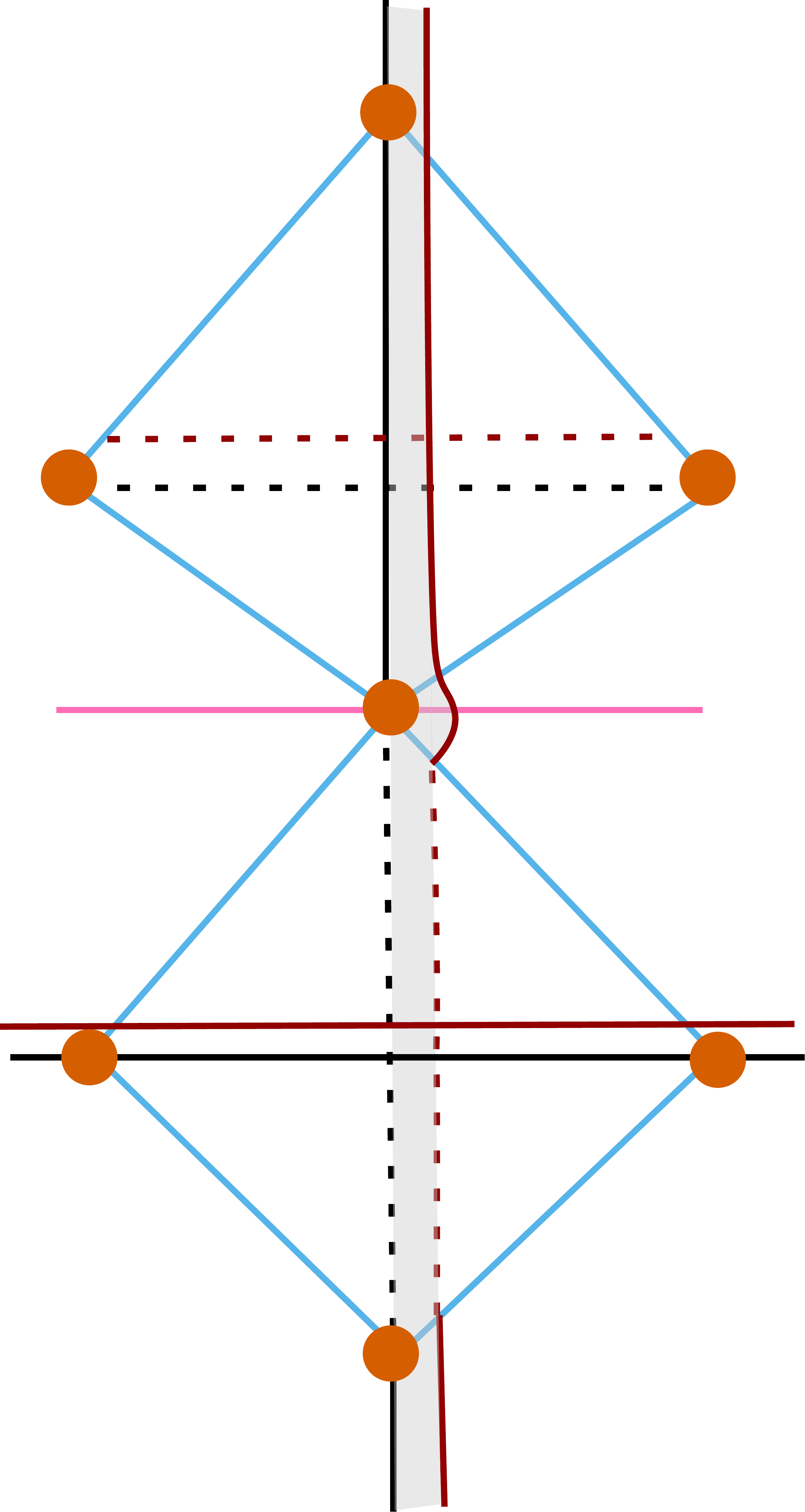}
\caption{A triangulation of a link exterior and the annulus in its $2$-skeleton.
}

\label{Figure:TriangulatingLinkExteriors}
\end{center}
\end{figure}
(for this construction see Figure~\ref{Figure:TriangulatingLinkExteriors}, where \(D\) is in black).  First we add a vertex in the middle of each
edge of \(G\) (orange disks in Figure~\ref{Figure:TriangulatingLinkExteriors}).
We connect the four new vertices around each crossing of \(L\) to form a
quadrilateral (light blue in Figure~\ref{Figure:TriangulatingLinkExteriors}).
Whenever two edges of \(G\) form a bigon, the two vertices in the middle of
these edges are connected with a \em single \em edge
(that is, a single light blue edge that is involved in two distinct quadrilaterals; see the top of Figure~\ref{Figure:TriangulatingLinkExteriors}).
Note that each quadrilateral, together with two subarcs of \(L\), forms the $1$-skeleton of a tetrahedron, and \(S\) intersects this tetrahedron in a disk whose boundary is the light blue quadrilateral.
Let \(\mathcal{T}'\) be the union of these tetrahedra.  The closure of each
  component of \(S \setminus \mathcal{T}'\) is an \(n\)-gon (for some \(n \geq
  3\); note that \(2\)-gons will not appear).  We subdivide each \(n\)-gon into
  \(n-2\) triangles, and add these triangles to \(\mathcal{T}'\).  This yields a \(3\)-complex
whose complement has two components, the closure
of each is a ball whose boundary inherits a triangulation.
Finally, we add a vertex in the interior of each
complementary ball, and cone the boundary of the ball to that vertex.
This yields a triangulation \(\mathcal{T}\) of
\(\mathbb{S}^3\) in which \(L\) embeds in the \(1\)-skeleton.  It is clear from
the construction that the time and number of tetrahedra needed are both linear in
\(c\).

Since \(L\) is embedded in the 1-skeleton of \(\mathcal{T}\), after two barycentric
subdivisions we obtain a triangulation in which  \(\mathrm{int}N(L)\),
an open neighborhood of \(L\), is embedded.  We remove
  \(\mathrm{int}N(L)\) $\mathbb{S}^3$ obtaining the desired
triangulation \(\mathcal{T}_{E(L)}\) of \(E(L)\).  Clearly, the time needed to construct this
triangulation and the number of tetrahedra are \(O(c)\).

It was shown in~\cite{hlp-ccklp-99} that the meridian is in the \(1\)-skeleton
of the boundary.  To see that the blackboard framing is embedded there as well,
note that \(k\) and a push off of \(k\) cobound an annulus in the
\(2\)-skeleton of the triangulation \(\mathcal{T}\) constructed above (see
Figure~\ref{Figure:TriangulatingLinkExteriors}, right).  It follows that the push off,
which is the blackboard framing, embeds in the \(1\)-skeleton of the boundary.
This completes the proof of Proposition~\ref{P:TriangulatingE(L)}.  \end{proof}

Our next goal is performing Dehn filling.  To that end we prove:

\begin{\lemm}\label{L:solidtorus}
Let $T$ be a triangulation of a torus with $n$ edges and $\gamma$ be a simple
cycle on the $1$-skeleton of $T$. Then one can compute a
triangulation of a solid torus $S$ in time $O(n)$ with $O(n)$ tetrahedra such that $\partial
S$ is simplicially isomorphic to $T$ and the image of $\gamma$ under this
isomorphism is a meridian of $S$.
\end{\lemm}

\begin{proof}

Starting from the triangulation $T$, we first attach a triangulated disk $D$ to \(\gamma\)
by adding one vertex \(u\) and coning the edges of \(\gamma\) to \(u\).  Note that \(D\) has $|\gamma|$
triangles.
Then we add a vertex $v$ and we form a cone with apex $v$
to every triangle of $T$, and also to every triangle of $D$ on both sides of
$D$.
This triangulation is not a simplicial complex yet
{(as it admits distinct \(3\)-simplices with the same vertex set)},
but can be made so by a
subdivision without changing the triangulation of $T$. The resulting
triangulation has $O(n)$ tetrahedra and can be built in linear time.
Topologically, we obtain a $3$-ball where two disks in the boundary have been identified to the single disk $D$. This yields a solid torus $S$, and the boundary of $S$ is by construction
simplicially isomorphic to $T$. Since $\gamma$ bounds the meridian disk $D$, it
is a meridian of $S$, which concludes the proof.
\end{proof}

In Section~\ref{S:reduction} we constructed, given \(\Phi \in \textsc{3-Sat}\),
a diagram of a link \(L\) in time \(O(|\Phi|^{2})\) and with  \(O(|\Phi|^{2})\)
crossings.  Proposition~\ref{P:TriangulatingE(L)} above yields a triangulation
\(\mathcal{T}\) of \(E(L)\) in time  \(O(|\Phi|^{2})\) and with
\(O(|\Phi|^{2})\) tetrahedra.  Let \(T_{i}\) be the boundary components
corresponding to the clasps, endowed with the triangulation induced by
\(\mathcal{T}_{E(L)}\).  By Proposition~\ref{P:TriangulatingE(L)}, the meridian
and the blackboard framing of each clasp are embedded in the \(1\)-skeleton of
the triangulation of \(T_{i}\); since the clasps have writhe \(0\) (recall Remark~\ref{remark:BBframing}~(5)),
the blackboard framing is, in fact, the longitude (Remark~\ref{remark:BBframing}~(3)).  Thus we see that the slopes \(1/0\) and \(0/1\) are embedded in the triangulation of \(T_{i}\), and by doing a constant number of barycentric subdivisions, we can refine $\mathcal{T}_{E(L)}$ so that $T_i$ contains, in its $1$-skeleton, a simple curve $\gamma_i$ realizing slope $3/2$. By Lemma~\ref{L:solidtorus}, we can compute a triangulation of a solid torus $S_i$ such that $\partial S_i$ is simplicially isomorphic to $T_i$ and the image of $\gamma_i$ under this isomorphism is a meridian of $S_i$. Now, gluing $S_i$ on $T_i$ yields a Dehn surgery corresponding to the surgery coefficient $3/2$.  By doing this for each clasp we obtain a triangulation of \(M(L)\).  In summary, we just proved:

\begin{cor}\label{C:computation}
A triangulation of $\mathbb{S}^3(L)$ with $O(|\Phi|^2)$ tetrahedra can be computed in time $O(|\Phi|^2)$.
\end{cor}


\section{Proof of Corollary~\ref{corollary}}
\label{section:ProofOfCorollary}

Recall that in the course of the proof of our main theorem we reduced
\textsc{3-Sat} to $\textsc{Embed}_{3\rightarrow3}$ by constructing, given a \textsc{3-SAT}
formula \(\Phi\) (satisfying certain conditions), a connected 3-manifold \(\mathbb{S}^{3}(L)\) that embeds in \(\mathbb{S}^{3}\) exactly when \(\Phi\) is satisfiable.  By construction, \(\mathbb{S}^{3}(L)\) is connected, orientable, and every boundary component of \(\mathbb{S}^{3}(L)\) is a torus, that is, \(\mathbb{S}^{3}(L) \in \mathcal{M}_{\mathrm{tor}}\).

To complete the proof of Corollary~\ref{corollary}, fix a triangulated closed
orientable irreducible 3-manifold \(M\) admitting no essential torus.  We need
to show that the theorem holds with \(M\) in the role of the range 3-manifold.
The proof assumes familiarity with irreducible manifolds, connected sums and
prime decompositions, and in particular the uniqueness of
decompositions. We refer to Hatcher~\cite[Section 1.1]{h-nb3mt-00} for the relevant background.

We first assume that \(M\) is connected.
A key tool in the proof of Theorem~\ref{T:main} was the Fox Re-embedding Theorem.  Re-embedding plays a key role in the proof of the corollary, and we will prove the necessary version in Lemma~\ref{lemma:reembading} below ({\it cf.}~\cite[Theorem~7]{ScharlemannThompsoFox}).

Now, given a \textsc{3-SAT} formula \(\Phi\), let \(\mathbb{S}^{3}(L)\) be the triangulated 3-manifold constructed in the proof of Theorem~\ref{T:main}.  Recall that \(\mathbb{S}^{3}(L)\) was constructed in polynomial time (in the size of \(\Phi\)).  After performing two barycentric subdivisions on \(\mathbb{S}^{3}(L)\) we obtain a triangulation that admits a tetrahedron \(T\) that is embedded in the interior of \(\mathbb{S}^{3}(L)\).  These subdivisions are done in linear time in the size of the triangulation of \(\mathbb{S}^{3}(L)\) and hence in polynomial time in the size of \(\Phi\).  Similarly, after performing a barycentric subdivision on \(M\) we obtain a triangulation that admits a tetrahedron \(T'\) that is embedded in the interior of \(M\).  This subdivision is done is constant time (since \(M\) is fixed).  We fix \(T\) and \(T'\).

Let \(X\) be the triangulated 3-manifold obtained from \(M\) and
\(\mathbb{S}^{3}(L)\) by removing the interiors of \(T\) and \(T'\) and
identifying their boundaries.  Topologically, we obtain the connected sum \(X \cong M \#
\mathbb{S}^{3}(L)\), so \(X\) contains a 2-sphere \(S^{*}\) (namely, \(\partial
T = \partial T'\)) that bounds, on one side, a punctured copy of \(M\), say
\(M^{*}\). 
We fix the notation \(S^{*}\) and \(M^{*}\) for the remainder of the proof.

Now we want to know that $X$ embeds into $M$ if and only if the
formula $\Phi$ is satisfiable. If $\Phi$ is satisfiable, then
$\mathbb{S}^{3}(L)$ embeds into $\mathbb{S}^{3}$ and therefore $X$ embeds into
$M \# \mathbb{S}^{3} \cong M$. For the other implication, we need the following
lemma.

\begin{\lemm}[Re-embedding]
\label{lemma:reembading}

Suppose that \(X\) is a manifold with toroidal boundary that embeds in an irreducible atoroidal manifold \(M\). Then, after re-embedding, we may assume that the components of \(\overline{M \setminus X}\) are solid tori.
\end{\lemm}

\begin{remark}
\label{remark:fox}
In the proof of Lemma~\ref{remark:fox} we allow \(M \cong \mathbb{S}^{3}\), giving a proof of the Fox Re-embedding Theorem when the boundary of the submanifold consists of tori, which is the only case used in this paper.
\end{remark}

First, we finish the proof of Corollary~\ref{corollary} assuming the lemma,
then we prove the lemma.

By Lemma~\ref{lemma:reembading} we have that \(M\) is obtained from \(X \cong M
\# \mathbb{S}^{3}(L)\) by Dehn filling.  Since \(M\) is closed, any Dehn
filling of \(X\) is done along components of \(\partial \mathbb{S}^{3}(L)\).
Thus the result of such a Dehn filling is \(M \# Y\), where \(Y\) is obtained
by Dehn filling \(\mathbb{S}^{3}(L)\).  Since \(M\) is irreducible, by uniqueness of
prime decompositions, we have that \(M \# Y \cong M\) if and only if \(Y \cong \mathbb{S}^{3}\). Thus \(M \# \mathbb{S}^{3}(L)\) embeds in \(M\) if and only if \(\mathbb{S}^{3}(L)\) embeds in \(\mathbb{S}^{3}\), and by our main theorem, this happens exactly when \(\Phi\) is satisfiable.

If \(M\) is not connected we apply the argument to one of its components; this completes the proof of the corollary assuming Lemma~\ref{lemma:reembading}.

\begin{proof}[Proof of Lemma~\ref{lemma:reembading}]
Since \(M\) is irreducible (or is \(\mathbb{S}^{3}\)) any sphere embedded in
\(M\) bounds a ball on (at least) one side, and on the other side it bounds a punctured copy of \(M\) (in other words, any sphere realizes the decomposition \(M \# \mathbb{S}^{3}\)).  We will often use this below.

Let \(T_{1},\dots,T_{n}\) be the boundary components of \(X\) and suppose that \(T_{1},\dots,T_{i-1}\) bound solid tori as required by the lemma.  We will re-embed \(X\) so that the \(T_{i}\) bounds a solid torus as well, and the lemma will follow by induction.   Let \(V_{i}\) be the component of  \(\overline{M \setminus X}\) with \(T_{i} \subset \partial V_{i}\).   Note that   \(T_{i} = \partial V_{i}\), for otherwise \(T_{i}\) is nonseparating and one of the following holds:
	\begin{enumerate}
	\item \(T_{i}\) is incompressible: this contradicts the assumption that \(M\) admits no essential torus.
	\item \(T_{i}\) is compressible: then the sphere obtained by compressing is nonseparating, contradicting irreducibility of \(M\).
	\end{enumerate}
Since \(M\) is admits no essential torus, \(T_{i}\) compresses and we consider two possibilities:
	\begin{enumerate}
	\item \(T_{i}\) compresses in \(V_{i}\): compressing \(T_{i}\) yields a sphere contained in \(V_{i}\) which bounds a ball \(B\).   If \(B \subset V_{i}\) then \(V_{i}\) is a solid torus (note that if \(M \cong \mathbb{S}^{3}\) this will always be the case; we now assume \(M \not\cong \mathbb{S}^{3}\)).  Otherwise, \(M^{*} \subset B\), which is impossible since \(M \not\cong \mathbb{S}^{3}\).

	\item \(T_{i}\) compresses in \(X\): let \(D\) be a compressing disk for \(T_{i}\) that intersects \(S^{*}\) transversely and minimizes \(\#(D \cap S^{*})\) among all such disks.  If  \(D \cap S^{*} \neq \emptyset\) then an easy Euler characteristic argument implies that some component of \(S^{*}\) cut open along  \(D \cap S^{*}\) is a disk
(an \em innermost \em disk),
and this disk allows us to cut and paste \(D\) in its interior, obtaining a compressing disk that intersects \(S^{*}\) fewer times than \(D\); thus  \(D \cap S^{*} = \emptyset\).   Let \(S'\) be the sphere obtained by compressing \(T_{i}\); by construction, \(S' \cap M_{1} = \emptyset\).  We claim that \(S'\) bounds a ball on the side containing \(V_{i}\).  If \(M \cong \mathbb{S}^{3}\) this is trivial, since \(S'\) bounds balls on both sides. We assume as we may that \(M \not\cong S^{3}\).  If the ball \(S'\) bounds is not on the side containing \(V_{i}\) then it contains \(M_{1}\), which is impossible since \(M \not\cong \mathbb{S}^{3}\).
Now we remove \(V_{i}\) and replace it with a solid torus (this is the re-embedding) so that the meridian of the solid torus intersected the boundary of \(D\) once; note that \(S'\) still bounds a ball in the side containing \(V_{i}\), so we did not change the underlying 3-manifold \(M\).
	\end{enumerate}
\end{proof}


\section*{Acknowledgment}
We would like to thank Ken Baker and Atsushi Ishii for useful
correspondence related to this work, and the anonymous reviewers
of the SODA version for helpful comments.

\bibliographystyle{alpha}
\bibliography{e3hard}

\end{document}